\documentclass[oneside,english]{amsart}
\usepackage[T1]{fontenc}
\usepackage[latin9]{inputenc}
\usepackage{geometry}
\usepackage{units}
\usepackage{mathrsfs}
\usepackage{mathtools}
\usepackage{amsbsy}
\usepackage{amstext}
\usepackage{amsthm}
\usepackage{amssymb}
\usepackage{stmaryrd}
\usepackage{stackrel}
\usepackage{graphicx}

\makeatletter

\newcommand{\noun}[1]{\textsc{#1}}
\providecommand{\tabularnewline}{\\}

\numberwithin{equation}{section}
\numberwithin{figure}{section}
\theoremstyle{plain}
\newtheorem{thm}{\protect\theoremname}
\theoremstyle{plain}
\newtheorem{prop}[thm]{\protect\propositionname}
\theoremstyle{remark}
\newtheorem{rem}[thm]{\protect\remarkname}
\theoremstyle{plain}
\newtheorem{lem}[thm]{\protect\lemmaname}
\theoremstyle{plain}
\newtheorem{cor}[thm]{\protect\corollaryname}

\usepackage{babel}
\usepackage{bbold}
\usepackage[all]{xy}

\makeatother

\usepackage{babel}
\providecommand{\corollaryname}{Corollary}
\providecommand{\lemmaname}{Lemma}
\providecommand{\propositionname}{Proposition}
\providecommand{\remarkname}{Remark}
\providecommand{\theoremname}{Theorem}

\begin{document}
\title{\noun{A minimal and non-alternative realisation of the Cayley plane}}
\author{Daniele Corradetti$^{*}$, Alessio Marrani $^{\dagger}$, Francesco
Zucconi $^{\ddagger}$}
\begin{abstract}
The compact 16-dimensional Moufang plane, also known as the Cayley
plane, has traditionally been defined through the lens of octonionic
geometry. In this study, we present a novel approach, demonstrating
that the Cayley plane can be defined in an equally clean, straightforward
and more economic way using two different division and composition
algebras: the paraoctonions and the Okubo algebra. The result is quite
surprising since paraoctonions and Okubo algebra possess a weaker
algebraic structure than the octonions, since they are non-alternative
and do not uphold the Moufang identities. Intriguingly, the real Okubo
algebra has $\text{SU}\left(3\right)$ as automorphism group, which
is a classical Lie group, while octonions and paraoctonions have an
exceptional Lie group of type $\text{G}_{2}$. This is remarkable,
given that the projective plane defined over the real Okubo algebra
is nevertheless isomorphic and isometric to the octonionic projective
plane which is at the very heart of the geometric realisations of
all types of exceptional Lie groups. Despite its historical ties with
octonionic geometry, our research underscores the real Okubo algebra
as the weakest algebraic structure allowing the definition of the
compact 16-dimensional Moufang plane. 
\end{abstract}

\maketitle
\tableofcontents{}

\section*{\noun{Introduction}}

A Moufang plane is a projective plane where every line is a translation
line or, alternatively, where the ``little Desargues theorem'' holds
(see in Sec. \ref{sec:Discussions-and-verifications}). Among the
various characteristics of Moufang planes, a notable one is their
dimensionality. Specifically, it is well-known that all compact, connected
Moufang planes are of dimension 2, 4, 8 and 16 and isomorphic to precisely
the projective planes over the Hurwitz division algebras $\mathbb{R},\mathbb{C},\mathbb{H}$
and $\mathbb{O}$. Of all these planes, the 16-dimensional Moufang
plane stands out due to the historical obstacles in its definition
arising from the lack of associativity of the octonions $\mathbb{O}$.
This definitional challenge sparked significant interest in mathematical
research during the early 20th century, culminating in one of the
most fascinating interplays between projective geometry, algebra,
and differential geometry. Indeed, one of the most remarkable achievements
of the resulting mathematical research activity, mainly due to Cartan
\cite{Car14}, Jordan, Wigner and von Neumann \cite{Jordan} and Freudenthal
\cite{Fr54}, is an interesting three-fold description of these planes:
as a completion of the affine plane $\mathscr{A}^{2}\left(\mathbb{K}\right)$,
for every $\mathbb{K}\in\left\{ \mathbb{R},\mathbb{C},\mathbb{H},\mathbb{O}\right\} $;
as the rank-1 idempotent elements of the rank-three Jordan algebra
$\mathfrak{J}_{3}\left(\mathbb{K}\right)$; as a coset manifold with
a specific isometry and isotropy group. Furthermore, the investigation
of octonionic geometry, particularly the study of the octonionic projective
plane $\mathbb{O}P^{2}$, unraveled a deep connection between octonions
and exceptional Lie groups \cite{Fr54,Freud 1965,Tits,Vinberg,Rosenf98}.
This connection, which was first envisaged by Cartan\cite{Car15}
and then explored by Chevalley and Schafer \cite{ChSch}, is so deep
that every known realization of compact exceptional Lie groups somehow
involves the octonions $\mathbb{O}$ in one form or another \cite{Yokota}.
Notably, each of these realizations of exceptional Lie groups has
a geometrical aspect in which the 16-dimensional Moufang plane plays
a pivotal role. Indeed, following Freudenthal \cite{Fr54,Freud 1965},
one can obtain all exceptional Lie groups of type $\text{F}_{4},\text{E}_{6},\text{E}_{7}$
and $\text{E}_{8}$ as transformation groups of the 16-dimensional
Moufang plane preserving the features of elliptic geometry, projective
geometry, symplectic geometry and metasymplectic geometry respectively
\cite{LM01}. 

Historically, the compact 16-dimensional Moufang plane's definition
arose out of octonionic geometry. However, in this work we show that
this plane can be defined in an equally clean, straightforward and
more minimal way by means of two different division composition algebras
endowed with less algebraic structure than the octonions, and that
do not uphold the Moufang identities, historically associated with
the Moufang property of the plane. 

Clearly, in order to define a 16-dimensional plane that satisfies
the affine and projective axioms of incidence geometry, an 8-dimensional
division algebra is necessary. Hurwitz theorem \cite{Hurwitz98} states
the existence of only one 8-dimensional division composition algebra
with a unit element, i.e. the algebra of octonions $\mathbb{O}$.
Yet, when non-unital algebras are considered, three 8-dimensional
division composition algebras emerge \cite{ElDuque Comp}: the aforementioned
octonions $\mathbb{O}$, the para-octonions $p\mathbb{O}$ (not to
be confused with the split-octonions that are not a division algebra)
and the real Okubo algebra $\mathcal{O}$. 

All three 8-dimensional algebras, being division and composition,
allow independent and self-contained definitions of an affine and
projective plane over them. Quite unexpectedly, despite the three
different algebraic origins, the three definitions give rise to the
same incidence plane: the compact 16-dimensional Moufang plane. This
result is quite surprising because the three algebras, though deeply
related, display very different properties. For instance, while octonions
have a unit element and paraoctonions have a paraunit, the Okubo algebra
merely contains idempotent elements. These differences apparently
show up into the projective planes defined over these algebras: e.g.,
as a consequence of not having an identity element, the points on
the Okubic plane $\left(0,0\right),\left(x,x\right)$ and $\left(y,y\right)$
are not all three incident to the same Okubic line, nor it exists
an Okubic collineation that switches coordinates, i.e. $\left(x,y\right)\longrightarrow\left(y,x\right)$,
as one has in octonionic case. Despite these apparent differences,
in Sec. \ref{sec:Three-realizations-of} we show that the projective
planes, obtained directly from their corresponding foundational algebras,
are all isomorphic and even all isometric one another. 
\begin{table}
\centering{}%
\begin{tabular}{|c|c|c|c|}
\hline 
Property & $\mathbb{O}$ & $p\mathbb{O}$ & $\mathcal{O}$\tabularnewline
\hline 
\hline 
Unital & Yes & No & No\tabularnewline
\hline 
Paraunital & Yes & Yes & No\tabularnewline
\hline 
Alternative & Yes & No & No\tabularnewline
\hline 
Flexible & Yes & Yes & Yes\tabularnewline
\hline 
Composition & Yes & Yes & Yes\tabularnewline
\hline 
Automorphism & $\text{G}_{2}$ & $\text{G}_{2}$ & $\text{SU}\left(3\right)$\tabularnewline
\hline 
\end{tabular}\caption{\label{tab:Synoptic-table-of}Synoptic table of the algebraic properties
of octonions $\mathbb{O}$, paraoctonions $p\mathbb{O}$ and the real
Okubic algebra $\mathcal{O}$.}
\end{table}

The result is remarkable in itself. However, since the 16-dimensional
Moufang plane is so deeply related with exceptional Jordan algebras,
exceptional Lie Groups and symmetric spaces, it also paves the way
for a novel, more minimal algebraic realization of these ubiquitous
mathematical objects. A synoptic summary of the algebraic properties
of octonions $\mathbb{O}$, paraoctonions $p\mathbb{O}$ and of the
real Okubic algebra $\mathcal{O}$ is summarized in Table \ref{tab:Synoptic-table-of}.
It is worth noting that the minimal algebraic structure between such
three algebras is the Okubo algebra $\mathcal{O}$, which is neither
unital, nor para-unital; it is non-alternative and has the smallest
automorphism group, i.e. $\text{SU}\left(3\right)$ which has dimension
8 compared to $G_{2}$ that is a 14-dimensional group. Both paraoctonions
$p\mathbb{O}$ and the real Okubic algebra $\mathcal{O}$ are non-alternative,
but flexible algebras. Their relation to the Moufang plane is thus
intriguing, because, notoriously, Moufang planes are associated to
Moufang identities, that in turn imply the alternativity of the underlying
algebra \cite{Mou35,HP}. In fact, all this does not give arise to
any contradiction, since both the Okubic and paraoctonionic projective
planes can be coordinatised by an alternative algebra, i.e. the octonions,
through a non-linear  planar ternary field as we show in Sec. \ref{sec:Discussions-and-verifications}. 

An even more striking observation is that, while the octonions possess
an automorphism group that is an exceptional group, the automorphism
Lie group of the real Okubo algebrais not exceptional, nor has any
immediate relation to $\text{G}_{2}$ itself. Nevertheless, the projective
plane over the Okubo algebra gives rise to a geometric realisation
of all types of exceptional Lie groups as $\text{G}_{2},\text{F}_{4},\text{E}_{6},\text{E}_{7}$
and $\text{E}_{8}$ as the transformation group respectively preserving:
the non-degenerate quadrangles of the plane (type $\text{G}_{2}$);
the distances of the plane (type $\text{F}_{4}$); the usual incidence
relations between line and points (type $\text{E}_{6}$); extended
incidence relations according to symplectic and metasymplectic geometry
(type $\text{E}_{7}$ and $\text{E}_{8}$, for this last part see
Freudenthal work \cite[Sec. 4.13]{Freud 1965}). It is well known
that all compact exceptional Lie groups have $\text{SU}\left(3\right)$
as a subgroup, this work points out how the presence of a subgroup
$\text{SU}\left(3\right)$ is related with an Okubic structure underlying
the 16-dimensional Moufang plane.

It is here worth recalling (see e.g. \cite{Ste80,Ste08}) that Lie
groups of type $\text{E}_{6}$ are largely studied and are still viable
candidates for GUT theories and that the real Okubo algebra was discovered
by Susumo Okubo in his investigations on $\text{SU}\left(3\right)$
as the gauge group for QCD \cite{Okubo95}. Thus, we expect the Okubic
formulation of the Cayley plane to find a physical application as
a concrete alternative to its octonionic realisation and to the octonionic
formulation of the rank-3 exceptional Jordan algebra, also known as
Albert algebra. Additionally, it is known that M-theory may display
an hidded Cayley-Moufang fibration \cite{Sa11}. Here it is worth
noting that variations in the foundational algebra of this plane could
potentially lead to novel physical theories.

The present work is thus structured as follows. In Sec. \ref{sec:Octonions,-para-octonions-and}
we review the three algebras we are going to use: octonions $\mathbb{O}$,
paraoctonions $p\mathbb{O}$ and the real Okubic algebra $\mathcal{O}$.
In Sec. \ref{sec:Affine-and-projective} we define the three affine
and projective planes. Since the construction is formally very similar
we develop only the details of the construction of Okubic affine and
projective plane, pointing out the differences occurring in the other
algebras. The main result is in Sec. \ref{sec:Three-realizations-of}
where we present the isomorphism between the three planes. Finally,
in Sec. \ref{sec:Discussions-and-verifications} we discuss our findings
and introduce a software tool that facilitates direct and numerical
verification of calculations involving octonionic, para-octonionic,
and Okubic computations. This tool has been made publicly available
and can be accessed on our GitHub repository at \texttt{https://github.com/DCorradetti/OkuboAlgebras}.

\section{\label{sec:Octonions,-para-octonions-and}Composition algebras }

Composition algebras are algebras endowed with a norm that enjoys
the multiplicative property, i.e. $n\left(x\cdot y\right)=n\left(x\right)n\left(y\right)$.
Composition algebras with multiplicative identity are called Hurwitz
algebras and are fully classified \cite{ElDuque Comp}. On the other
hand, composition algebras without multiplicative identity but with
associative norm were discovered by Petersson \cite{Petersson 1969}
and indipendently by Okubo \cite{Okubo 1978}; they are now called
symmetric composition algebras \cite{KMRT} and are completely classified
in para-Hurwitz and Okubo algebras \cite{ElDuque Comp}. Para-Hurwitz
algebras are non-unital composition algebras strictly related to their
unital companion, i.e. the corresponding Hurwitz algebra, while on
the other hand Okubo algebras are somewhat more unique in feature
appearing only as 8-dimensional algebras and with some very peculiar
characteristics that distinguish them from both Huwitz and para-Hurwitz
algebras. It is worth noting that while it is possible to define an
Okubo algebra over any field, here we will be focusing on the Okubo
algebra over the real $\mathbb{R}$, which is a division composition
algebra. 

In this section we review some useful notions about composition algebras.
Then we focus on Hurwitz algebras and, subsequently, we enter into
the realm of symmetric composition algebras, specifically highlighting
para-Hurwitz and Petersson algebras that in fact exhaust all algebras
of this family. Even if this section is made of known results, we
thought it might be worthwhile to collect them in a few pages of review
content given their paramount importance in the understanding of the
algebraic context of the subsequent sections.

\subsection{Composition Algebras}

An\emph{ algebra}, denoted by $A$, is a vector space over a field
$\mathbb{F}$ equipped with a bilinear multiplication. For our discussion,
we will restrict our attention to algebras of finite dimension and
the field $\mathbb{F}$ will be taken to be either the field of real
$\mathbb{R}$ or complex numbers $\mathbb{C}$. The specific properties
of the multiplication operation in an algebra lead to various classifications.
Specifically, an algebra $A$ is said to be\emph{ commutative} if
$x\cdot y=y\cdot x$ for every $x,y\in A$; is \emph{associative}
if satisfies $x\cdot\left(y\cdot z\right)=\left(x\cdot y\right)\cdot z$;
is \emph{alternative} if $x\cdot\left(y\cdot y\right)=\left(x\cdot y\right)\cdot y$;
and finally, \emph{flexible} if $x\cdot\left(y\cdot x\right)=\left(x\cdot y\right)\cdot x$.
It is worth noting that the last three proprieties can be seen as
successive refinements of associativity, i.e.
\begin{equation}
\text{associative}\Rightarrow\text{alternative}\Rightarrow\text{flexible}.
\end{equation}
This observation stems from a nontrivial theorem proved by Artin (see
\cite{Scha}) who showed that all alternative algebras are flexible. 

Since $A$ must be a group with respect to addition, every algebra
has a zero element $0\in A$. Furthermore, if the algebra does not
have zero divisors, it is referred to as a \emph{division} algebra,
i.e. an algebra for which $x\cdot y=0$ implies $x=0$ or $y=0$.
While the zero element is a universal feature in any algebra, the
algebra is termed \emph{unital} if there exists an element $1\in X$
such that $1\cdot x=x\cdot1=x$ for all $x\in A$. 

Consider an algebra $A$. Then a quadratic form $n$ on $A$ over
the field $\mathbb{F}$ is called \emph{norm} and its polarization
is given by 
\begin{equation}
\left\langle x,y\right\rangle =n\left(x+y\right)-n\left(x\right)-n\left(y\right),\label{eq:polarNorm}
\end{equation}
so that the norm can be explicitly given as
\begin{equation}
n\left(x\right)=\frac{1}{2}\left\langle x,x\right\rangle ,\label{eq:n(x)=00003D1/2(x,x)}
\end{equation}
for every $x\in A$. An algebra $A$ with a non-degenerate norm $n$
that satisfies the following multiplicative property, i.e.

\begin{align}
n\left(x\cdot y\right) & =n\left(x\right)n\left(y\right),\label{eq:comp(Def)}
\end{align}
for every $x,y\in A$, is called a \emph{composition} algebra and
is denoted with the triple $\left(A,\cdot,n\right)$ or simply as
$A$ if there are no reason for ambiguity.

Given a composition algebra $A$, applying equation (\ref{eq:polarNorm})
to the multiplicative property of the norm expressed in (\ref{eq:comp(Def)}),
we find that

\begin{align}
\left\langle x\cdot y,x\cdot z\right\rangle  & =n\left(x\right)\left\langle y,z\right\rangle ,
\end{align}
for every $x,y,z\in A$, which is an useful identity to be aware of. 

\subsection{Unital composition algebras}

Composition algebras that possess an unit element are called \emph{Hurwitz
algebras}. The interplay between the multiplicative property of the
norm in (\ref{eq:comp(Def)}) and the existence of a unit element,
is full of interesting implications. Indeed, every Hurwitz algebra
is endowed with an order-two antiautomorphism called \emph{conjugation},
defined by 
\begin{equation}
\overline{x}=\left\langle x,1\right\rangle 1-x.\label{eq:conjugation}
\end{equation}
 The linearization of the norm, when paired with the composition,
results in the notable relation $\left\langle x\cdot y,z\right\rangle =\left\langle y,\overline{x}\cdot z\right\rangle ,$
that imply that $\overline{x\cdot y}=\overline{y}\cdot\overline{x}$
and 
\begin{equation}
x\cdot\overline{x}=n\left(x\right)1.\label{eq:ConjugNorm}
\end{equation}
Moreover, from the existence of a unit element in a composition algebra
we have that elements with unit norm form a goup and, even more strikingly,
that the whole algebra must be alternative (for a proof see \cite[Prop. 2.2]{ElDuque Comp}). 

Equation (\ref{eq:ConjugNorm}) can be rephrased in the well-known
\emph{Hamilton-Cayley equation,} $x^{2}-\left\langle x,1\right\rangle x-n\left(x\right)1=0,$
which holds true for every unital composition algebra. Finally, a
relation that is crucial for the Veronesean representation of the
projective plane over a unital composition algebras, is the following
\begin{equation}
x\cdot\left(\overline{x}\cdot y\right)=\left(x\cdot\overline{x}\right)\cdot y=n\left(x\right)y,\label{eq:compAlg x.x=0000BA.y=00003Dn(x)y}
\end{equation}
which has a nice analogous in the case of \emph{symmetric composition}
algebras that we discuss in Section \ref{sec:Symmetric-composition-algebras}. 

A major theorem by Hurwitz proves that the only unital composition
algebras over the reals are $\mathbb{R},\mathbb{C},\mathbb{H}$ and
$\mathbb{O}$ accompanied by their split counterparts $\mathbb{C}_{s},\mathbb{H}_{s},\mathbb{O}_{s}$
(see \cite[Cor. 2.12]{Hurwitz98,ElDuque Comp}). Consequently, there
are seven Hurwitz algebras, each having real dimensions of 1, 2, 4,
or 8. Out of these, four are also division algebras, i.e. $\mathbb{R},\mathbb{C},\mathbb{H}$
and $\mathbb{O}$, while three are split algebras and thus have zero
divisors, i.e. $\mathbb{C}_{s},\mathbb{H}_{s},\mathbb{O}_{s}$. The
properties of such algebras are quite different one another. More
specifically, $\mathbb{R}$ is also totally ordered, commutative and
associative; $\mathbb{C}$ is just commutative and associative; $\mathbb{H}$
is only associative and, finally, $\mathbb{O}$ is only alternative.
\begin{table}
\begin{centering}
\begin{tabular}{|c|c|c|c|c|c|c|c|c|c|c|c|c|}
\hline 
\textbf{Hurwitz} & \textbf{O.} & \textbf{C.} & \textbf{A.} & \textbf{Alt.} & \textbf{F.} &  & \textbf{p-Hurwitz} & \textbf{O.} & \textbf{C.} & \textbf{A.} & \textbf{Alt.} & \textbf{F.}\tabularnewline
\hline 
\hline 
$\mathbb{R}$ & Yes & Yes & Yes & Yes & Yes &  & $p\mathbb{R}\cong\mathbb{R}$ & Yes & Yes & Yes & Yes & Yes\tabularnewline
\hline 
$\mathbb{C}$, $\mathbb{C}_{s}$ & No & Yes & Yes & Yes & Yes &  & $p\mathbb{C}$, $p\mathbb{C}_{s}$ & No & Yes & No & No & Yes\tabularnewline
\hline 
$\mathbb{H}$,$\mathbb{H}_{s}$ & No & No & Yes & Yes & Yes &  & $p\mathbb{H}$,$p\mathbb{H}_{s}$ & No & No & No & No & Yes\tabularnewline
\hline 
$\mathbb{O}$,$\mathbb{O}_{s}$ & No & No & No & Yes & Yes &  & $p\mathbb{O}$,$p\mathbb{O}_{s}$ & No & No & No & No & Yes\tabularnewline
\hline 
\end{tabular}{\small{} }{\small\par}
\par\end{centering}
\begin{centering}
{\small{}\bigskip{}
}\caption{\emph{\label{tab:Hurwitz-para-Hurwitz}On the left,} we have summarized
the algebraic properties, i.e. totally ordered (O), commutative (C),
associative (A), alternative (Alt), flexible (F), of all Hurwitz algebras,
namely $\mathbb{R},\mathbb{C},\mathbb{H}$ and $\mathbb{O}$ along
with their split counterparts $\mathbb{C}_{s},\mathbb{H}_{s},\mathbb{O}_{s}$.
\emph{On the right}, we have summarized the algebraic properties of
all para-Hurwitz algebras, namely $p\mathbb{R},p\mathbb{C},p\mathbb{H}$
and $p\mathbb{O}$ accompanied by their split counterparts $p\mathbb{C}_{s},p\mathbb{H}_{s},p\mathbb{O}_{s}$.}
\par\end{centering}
\end{table}
 As shown by Table \ref{tab:Hurwitz-para-Hurwitz} all properties
of $\mathbb{R},\mathbb{C},\mathbb{H}$ and $\mathbb{O}$ are valid
also for the split companions with the only difference that the latter
are not division algebras and do have zero divisors. Generalizations
of Hurwitz Theorem can be done over arbitrary fields (see \cite[p. 32]{ZSSS})
but for our purposes this will not be needed.

\subsection{\label{sec:Symmetric-composition-algebras}Symmetric composition
algebras}

We now turn our attention to a special class of composition algebras,
i.e. symmetric composition algebras, that are not unital but exhibit
many properties analogous of Hurwitz algebras. Compositions algebras
with associative norms (see below) were independently studied by Petersson
\cite{Petersson 1969}, Okubo \cite{Okubo95}, and Faulkner \cite{Fau14}.
In \cite{OO81a}, Okubo-Osborn shown that over an algebraically closed
field the only two types of symmetric composition algebras are para-Hurwitz
algebras and Okubo algebras, but a final classification was done by
Elduque and Myung \cite{Elduque 91,Elduque Myung 90}.

A symmetric composition algebra $\left(A,*,n\right)$ is a composition
algebra wherein the norm is associative, i.e. satisfies the identity
\begin{equation}
\left\langle x*y,z\right\rangle =\left\langle x,y*z\right\rangle ,\label{eq:associativityNorm}
\end{equation}
where $x,y,z\in A$ and $\left\langle x,y\right\rangle =n\left(x+y\right)-n\left(x\right)-n\left(y\right)$,
as stated in (\ref{eq:polarNorm}). 

From equation (\ref{eq:associativityNorm}), we extract a significant
attribute of symmetric composition algebras. More precisely, considering:
\begin{align}
\left\langle \left(x*y\right)*x,z\right\rangle  & =\left\langle x*y,x*z\right\rangle =n\left(x\right)\left\langle y,z\right\rangle ,
\end{align}
and given that $n\left(x+y\right)=n\left(x\right)+n\left(y\right)+\left\langle x,y\right\rangle $,
we can deduce
\begin{align}
n\left(\left(x*y\right)*x-n\left(x\right)y\right) & =2n^{2}\left(x\right)n\left(y\right)-n\left(x\right)\left\langle x*y,x*y\right\rangle =0.
\end{align}
Thus, since the norm $n$ is non singular we have the following important
proposition
\begin{prop}
\label{prop:x*y*x}Let $\left(A,*,n\right)$ be symmetric composition
algebra then 
\begin{equation}
\left(x*y\right)*x=n\left(x\right)*y,\label{eq:x*y*x=00003Dn(x)y}
\end{equation}
for every $x,y\in A$.
\end{prop}

In the realm of Hurwitz algebras, and similarly for symmetric composition
algebras, all automorphisms are isometries. Indeed, it sufficies to
consider that a map $\varphi:A\longrightarrow A$ such that $\varphi\left(x*y\right)=\varphi\left(x\right)*\varphi\left(y\right),$implies
that $\varphi\left(\left(x*y\right)*x\right)=n\left(x\right)*\varphi\left(y\right),$
on one side, while on the other hand, $\left(\varphi\left(x\right)*\varphi\left(y\right)\right)*\varphi\left(x\right)=n\left(\varphi\left(x\right)\right)*\varphi\left(y\right),$
so that it must be
\begin{equation}
n\left(\varphi\left(x\right)\right)=n\left(x\right),
\end{equation}
for every $x\in A$.

In fact, symmetric composition algebras are deeply intertwined with
Hurwitz algebras. Indeed, given a symmetric composition algebra $\left(A,*,n\right)$
and a norm $1$ element $a\in A$, we can utilize Kaplansky\textquoteright s
trick to define a new product
\begin{equation}
x\cdot y=\left(a*x\right)*\left(y*a\right),
\end{equation}
for every $x,y\in A$, resulting in a new composition algebra $\left(A,\cdot,n\right)$.
Now, consider the element $e=a*a$. Since (\ref{eq:x*y*x=00003Dn(x)y})
and $n\left(a\right)=1$ we then have that 
\begin{align}
e\cdot x & =\left(a*\left(a*a\right)\right)*\left(x*a\right)=x,\\
x\cdot e & =\left(a*x\right)*\left(\left(a*a\right)*a\right)=x,
\end{align}
for every $x\in A$. Consequently, $\left(A,\cdot,n\right)$ is a
unital composition algebra, or equivalently, a Hurwitz algebra. As
a direct implication of the Hurwitz theorem, symmetric composition
algebras can only have dimensions of 1, 2, 4, or 8. 

\subsubsection{Para-Hurwitz algebras }

An important class of symmetric composition algebras is that of para-Hurwitz
algebras. Given any Hurwitz algebra $\left(A,\cdot,n\right)$ a conjugation
is naturally defined as 
\begin{equation}
\overline{x}=\left\langle x,1\right\rangle 1-x,
\end{equation}
for every $x\in A$. Then, consider the new product
\begin{equation}
x\bullet y=\overline{x}\cdot\overline{y},\label{eq:para-Hurwitz}
\end{equation}
for every $x,y\in A$. Since $n\left(x\right)=n\left(\overline{x}\right)$
we have that 
\begin{equation}
n\left(x\bullet y\right)=n\left(\overline{x}\cdot\overline{y}\right)=n\left(x\right)n\left(y\right),
\end{equation}
and thus the algebra $\left(A,\bullet,n\right)$ is again a composition
algebra. On the other hand $\left(A,\bullet,n\right)$ is not an unital
algebra since 
\begin{equation}
x\bullet1=1\bullet x=\overline{x}.
\end{equation}
Moreover, the algebra is a symmetric composition algebra since it
can be shown to upholds 
\begin{align}
\left\langle x\bullet y,z\right\rangle  & =\left\langle x,y\bullet z\right\rangle ,
\end{align}
and it is then called a \emph{para-Hurwitz} algebra \cite{ElDuque Comp}.
For every Hurwitz algebra, i.e. unital composition algebra, of dimension
$>1$ we have a para-Hurwitz algebra that is a symmetric composition
algebra that we denote as $p\mathbb{C},$$p\mathbb{C}_{s}$,$p\mathbb{H},$$p\mathbb{H}_{s}$,
$p\mathbb{O}$ and $p\mathbb{O}_{s}$ respectively. It is worth noting
that all para-Hurwitz algebras are non-alternative algebras, since
\begin{align}
x\bullet\left(x\bullet y\right) & =\overline{x}\cdot\left(\overline{\overline{x}\cdot\overline{y}}\right)=\overline{x}\cdot\left(y\cdot x\right),\\
\left(x\bullet x\right)\bullet y & =\left(x\cdot x\right)\cdot\overline{y},
\end{align}
thus, in general, $x\bullet\left(x\bullet y\right)\neq\left(x\bullet x\right)\bullet y$.
Nevertheless, by Proposition \ref{prop:x*y*x} they are flexible and
more specifically 
\begin{equation}
x*y*x=n\left(x\right)*y,
\end{equation}
for every $x,y\in A$. Moreover, if the the Hurwitz algebra $\left(A,\cdot,n\right)$
is a division algebra, then also the para-Hurwitz $\left(A,*,n\right)$
defined from (\ref{eq:para-Hurwitz}) is a division algebra. Algebraic
properties of the Hurwitz algebras are summarized in Table \ref{tab:Hurwitz-para-Hurwitz}.

\subsubsection{Petersson algebras }

A generalisation of para-Hurwitz algebras was presented by Petersson
in \cite{Petersson 1969}. Starting with a Hurwitz algebra $\left(A,\cdot,n\right)$,
he introduced a new algebra $\left(A,*,n\right)$ such that 
\begin{equation}
x*y=\tau\left(\overline{x}\right)\cdot\tau^{2}\left(\overline{x}\right),
\end{equation}
where $\tau$ is an order three automorphisms, i.e. $\tau^{3}=\text{id}$.
The new algebra, typically denoted as $A_{\tau}$, becomes a composition
algebra that is non-unital. Moreover, Petersson demonstrated that
over an algebraically closed field like $\mathbb{C}$, there exists
a specific automorphism that results in a non-para-Hurwitz algebra.
This new algebra is a symmetric composition algebra containing idempotent
elements.

Petersson algebras are crucial in characterizing symmetric composition
algebras since we have the following 
\begin{thm}
\emph{\label{thm:(Elduque-Perez-)-An}(Elduque-Perez \cite[Th. 2.5]{EP96})
}An algebra $A$ is a symmetric composition algebra with an nonzero
idempotent if and only if there exists a Hurwitz algebra $H$ and
an automorphisms $\tau$ of $H$ such that $A$ is isomorphic to the
algebra $H_{\tau}$.
\end{thm}

\section{\label{sec:Octonions,-Paraoctonions-and}Octonions, paraoctonions
and the real Okubo algebra}

In this section, we delve into the three division algebras of primary
interest: the octonions $\mathbb{O}$, the paraoctonions $p\mathbb{O}$
and the real Okubo algebra $\mathcal{O}$. The main objective of this
section is summarized in Table \ref{tab:Oku-Para-Octo} that synoptically
illustrates the relationships between the product of the three algebras.
It's crucial to note that although it is possible to switch from one
algebra to another by altering the product definition, none of these
algebras are isomorphic to the others: e.g. the octonions $\mathbb{O}$
are alternative and unital, para-octonions $p\mathbb{O}$ are nor
alternative nor unital but do have a para-unit, while the Okubo algebra
$\mathcal{O}$ is non-alternative and only has idempotents elements.
It's also worth highlighting that the Okubo algebra $\mathcal{O}$
is the least structured among these algebras.

\subsection{The algebra of octonions}

The algebra of octonions $\mathbb{O}$ is the only division Hurwitz
algebra with a dimension of eight. We define the composition algebra
of octonion $\left(\mathbb{O},\cdot,n\right)$ as the eight dimensional
real vector space with basis $\left\{ \text{i}_{0}=1,\text{i}_{1},...,\text{i}_{7}\right\} $
with a bilinear product encoded through the \emph{Fano plane} and
explained in Fig. \ref{fig:octonion fano plane}. 
\begin{figure}
\centering{}\includegraphics[scale=0.08]{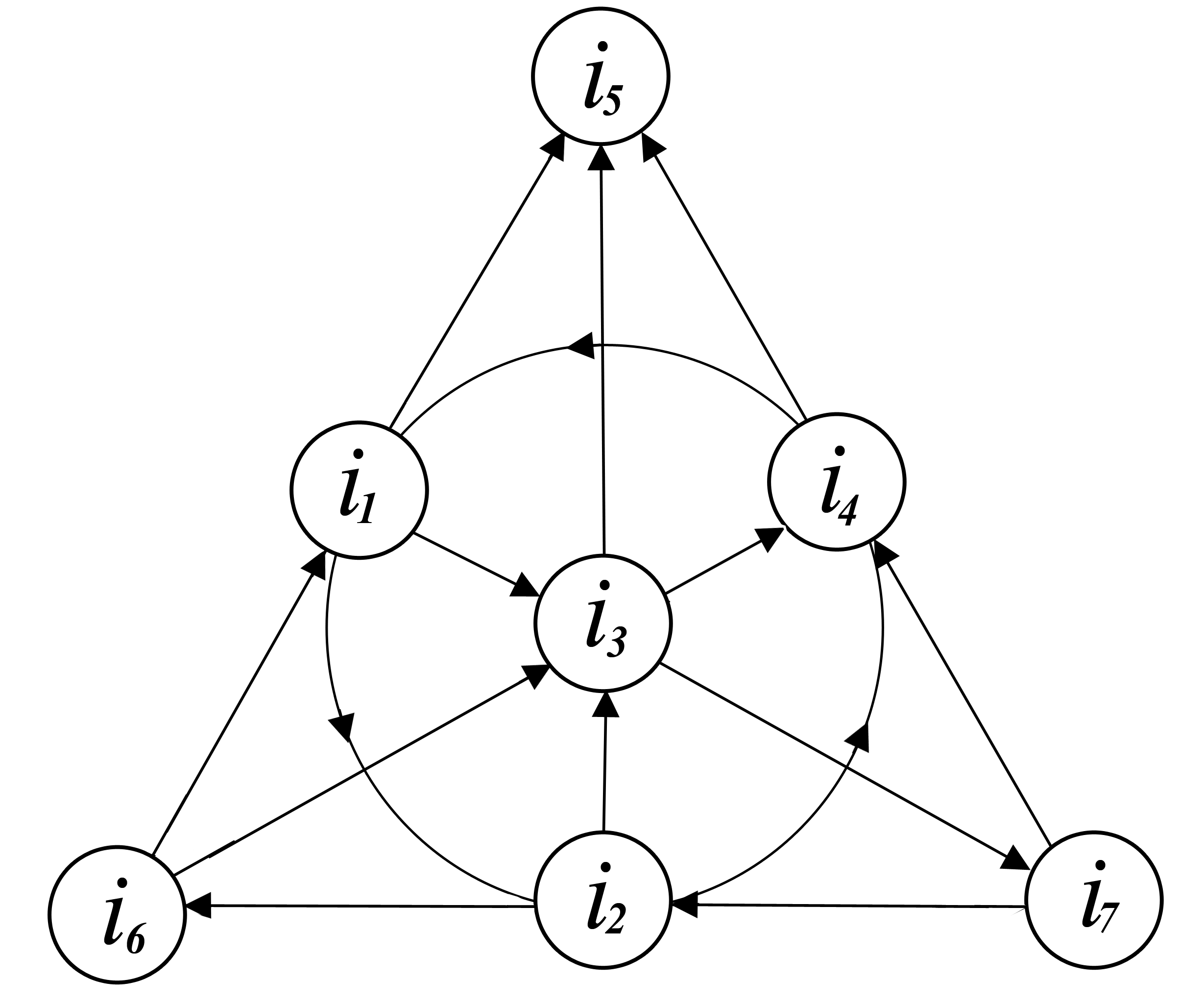}\caption{\label{fig:octonion fano plane} Multiplication rules for octonions
$\mathbb{O}$ as real vector space $\mathbb{R}^{8}$ in the basis
$\left\{ \text{i}_{0}=1,\text{i}_{1},...,\text{i}_{7}\right\} $.
Lines in the Fano plane identify associative triples of the product
and the arrow indicates the sign (positive in the sense of the arrow
and negative in the opposite sense). In addition to the previous rules
it is intended that $\text{i}_{k}^{2}=-1$.}
\end{figure}

Given an element $x\in\mathbb{O}$ with decomposition 
\begin{equation}
x=x_{0}+\stackrel[k=1]{7}{\sum}x_{k}\text{i}_{k},
\end{equation}
the norm $n$ is the obvious Euclidean one defined by 
\begin{equation}
n\left(x\right)=x_{0}^{2}+x_{1}^{2}+x_{2}^{2}+x_{3}^{2}+x_{4}^{2}+x_{5}^{2}+x_{6}^{2}+x_{7}^{2},\label{eq:octonionic norm}
\end{equation}
 for which the conjugation results 
\begin{equation}
\overline{x}=x_{0}-\stackrel[k=1]{7}{\sum}x_{k}i_{k},
\end{equation}
 and therefore 
\begin{equation}
n\left(x\right)=\overline{x}\cdot x,\label{eq:n(x)=00003Dxcx}
\end{equation}
as it happens for every Hurwitz algebra. Then a look at (\ref{eq:octonionic norm})
shows that $n\left(x\right)=0$ if and only if $x=0$ and thus the
inverse of a non-zero element of the octonions is easily found as
\begin{equation}
x^{-1}=\frac{\overline{x}}{n\left(x\right)}.
\end{equation}
 Also, from (\ref{eq:n(x)=00003Dxcx}) we have that the octonionic
inner product is given by 
\begin{align}
\left\langle x,y\right\rangle  & =x\overline{y}+y\overline{x},\label{eq:octonionic inner}
\end{align}
so that $\left\langle x,x\right\rangle =2n\left(x\right)$.

Straightforward calculations shows that the algebra of octonions is
neither commutative nor associative, but it is alternative. But, since
any two elements of an alternative algebra generate an associative
subalgebra, it is then easy to see that $\left(\mathbb{O},\cdot,n\right)$
is indeed an Hurwitz algebra since it is unital and 
\begin{align}
n\left(x\cdot y\right) & =\left(\overline{x\cdot y}\right)\cdot\left(x\cdot y\right)\\
 & =\left(\overline{y}\cdot\overline{x}\right)\cdot\left(x\cdot y\right)\nonumber \\
 & =\overline{y}\cdot\left(\overline{x}\cdot x\right)\cdot y=n\left(x\right)n\left(y\right).\nonumber 
\end{align}
Since the algebra is a composition algebra and any non-zero element
has non-zero norm, i.e. $n\left(x\right)\neq0$ then $\left(\mathbb{O},\cdot,n\right)$
is also a division algebra since if $x\cdot y=0$ then 
\begin{equation}
n\left(x\cdot y\right)=n\left(x\right)n\left(y\right)=0,
\end{equation}
which implies that $x=0$ or $y=0$. Moreover, an important relation
that will be used later on in the definition of the projective plane
is the following consequence of alternativity, i.e.
\begin{equation}
\overline{x}\cdot\left(x\cdot y\right)=n\left(x\right)y,\label{eq:Oct-xb*x*y=00003Dn(x)y}
\end{equation}
 for every $x,y\in\mathbb{O}$. 

\subsubsection{Moufang identities}

While octonions are not a group under multiplication due to their
lack of associativity, non-zero octonions form a \emph{Moufang loop},
i.e. a loop that satisfy\emph{ }the following \emph{Moufang identities},
i.e.
\begin{align}
\left(\left(x\cdot y\right)\cdot x\right)\cdot z & =x\cdot\left(y\cdot\left(x\cdot z\right)\right),\\
\left(\left(z\cdot x\right)\cdot y\right)\cdot x & =z\cdot\left(x\cdot\left(y\cdot x\right)\right),\\
\left(x\cdot y\right)\cdot\left(z\cdot x\right) & =x\cdot\left(\left(y\cdot z\right)\cdot x\right),\label{eq:MoufangIdent3}
\end{align}
for every $x,y,z\in\mathbb{O}$. Moufang identities are particularly
relevant since they are historically linked to geometrical properties
of the Moufang plane (see Sec. \ref{sec:Discussions-and-verifications}).
It is worth noting that any unital algebra satisfying Moufang identities
is an alternative algebra. Indeed, setting $z=1$ Moufang identities
turn into the flexible identity, i.e.
\begin{equation}
\left(x\cdot y\right)\cdot x=x\cdot\left(y\cdot x\right),
\end{equation}
 while setting $y=1$ we have the identity for the left and right
alternativity, i.e.
\begin{align}
\left(x\cdot x\right)\cdot z & =x\cdot\left(x\cdot z\right),\\
\left(z\cdot x\right)\cdot x & =z\cdot\left(x\cdot x\right).
\end{align}
 Thus, non alternative algebras do not uphold Moufang identities.

\subsection{Okubo algebras}

Symmetric composition algebras might have remained relatively unnoticed
among algebraists had Petersson \cite{Petersson 1969} not demonstrated
that for every field $\mathbb{F}$ there exists a unique eight-dimensional
algebra that is not a para-Hurwitz algebra. This result essentially
broadened the reach of the Hurwitz classification theorem. On the
other hand, Okubo algebras were independently developed by mathematical
physicist Susumo Okubo in the course of his work on quarks and Gell-Mann
matrices while pursuing an algebra that featured $\text{SU}\left(3\right)$
as automorphism group instead of $\text{G}_{2}$ as in the case of
Octonions\cite{Okubo95}. Even more interestingly, Okubo discovered
that such algebra is a division composition algebra and a deformation
of its product would give back the octonions\cite{Okubo 1978,Okubo 78c,Okubo1978b}.
It was with more recent works \cite{KMRT}, with the joint efforts
of Osborn, Elduque and Myung \cite{OO81a,OO81b,Elduque 91,Elduque 93,Elduque Myung 90,ElduQueAut},
that the context of Okubo algebras was fully elucidated. 

Following \cite{Okubo 1978} and \cite{Elduque Myung 90}, we define
the real Okubo Algebra $\mathcal{O}$ as the set of three by three
Hermitian traceless matrices over the complex numbers $\mathbb{C}$
with the following bilinear product 
\begin{equation}
x*y=\mu\cdot xy+\overline{\mu}\cdot yx-\frac{1}{3}\text{Tr}\left(xy\right),\label{eq:product Ok}
\end{equation}
where $\mu=\nicefrac{1}{6}\left(3+\text{i}\sqrt{3}\right)$ and the
juxtaposition is the ordinary associative product between matrices.
It is worth noting that (\ref{eq:product Ok}) can be seen as a modification
of the Jordanian product. Indeed, setting $\mu=\nicefrac{1}{2}$ and
negletting the last term, we retrieve the usual Jordan product over
Hermitian traceless matrices, i.e.
\begin{equation}
x\circ y=\frac{1}{2}xy+\frac{1}{2}yx.
\end{equation}
Nevertheless, Hermitian traceless matrices are not closed under such
product, thus requiring the additional term $-\nicefrac{1}{3}\text{Tr}\left(xy\right)$
for the closure of the algebra. Indeed, setting in (\ref{eq:product Ok})
$\text{Im}\mu=0$, one retrieves from the traceless part of the exceptional
Jordan algebra $\mathfrak{J}_{3}\left(\mathbb{C}\right)$, whose derivation
Lie algebra is $\mathfrak{su}\left(3\right)$.

Analyzing (\ref{eq:product Ok}), it becomes evident that the resulting
algebra is neither unital, associative, nor alternative. Nonetheless,
$\mathcal{O}$ is a\emph{ flexible }algebra, i.e. 
\begin{equation}
x*\left(y*x\right)=\left(x*y\right)*x,
\end{equation}
which will turn out to be an even more useful property than alternativity
in the definition of the projective plane. Even though the Okubo algebra
is not unital, it does have idempotents, i.e. $e*e=e$, such as 
\begin{equation}
e=\left(\begin{array}{ccc}
2 & 0 & 0\\
0 & -1 & 0\\
0 & 0 & -1
\end{array}\right),\label{eq:idemp}
\end{equation}
that together with
\begin{equation}
\begin{array}{ccc}
\text{i}_{1}=\sqrt{3}\left(\begin{array}{ccc}
0 & 1 & 0\\
1 & 0 & 0\\
0 & 0 & 0
\end{array}\right), &  & \text{i}_{2}=\sqrt{3}\left(\begin{array}{ccc}
0 & 0 & 1\\
0 & 0 & 0\\
1 & 0 & 0
\end{array}\right),\\
\text{i}_{3}=\sqrt{3}\left(\begin{array}{ccc}
0 & 0 & 0\\
0 & 0 & 1\\
0 & 1 & 0
\end{array}\right), &  & \text{i}_{4}=\sqrt{3}\left(\begin{array}{ccc}
1 & 0 & 0\\
0 & -1 & 0\\
0 & 0 & 0
\end{array}\right),\\
\text{i}_{5}=\sqrt{3}\left(\begin{array}{ccc}
0 & -i & 0\\
i & 0 & 0\\
0 & 0 & 0
\end{array}\right), &  & \text{i}_{6}=\sqrt{3}\left(\begin{array}{ccc}
0 & 0 & -i\\
0 & 0 & 0\\
i & 0 & 0
\end{array}\right),\\
\text{i}_{7}=\sqrt{3}\left(\begin{array}{ccc}
0 & 0 & 0\\
0 & 0 & -i\\
0 & i & 0
\end{array}\right),
\end{array}\label{eq:definizione i ottonioniche}
\end{equation}
form a basis for $\mathcal{O}$ that has real dimension $8$.\footnote{Actually, the 8 matrices three by three (2.19)-(2.20) are, up to an
overall factor $\sqrt{3}$, the Gell-Mann matrices. In particular,
the idempotent (2.19) is $-\sqrt{3}$ times the eighth Gell-Mann matrix
$\lambda_{8}$,with the first and third rows and columns exchanged.} It is worth noting that the choice of the idempotent $e$ as in (\ref{eq:idemp})
does not yield to any loss of generality for the subsequent development
of our work since all idempotents are conjugate under the automorphism
group (cfr. \cite[Thm. 20]{ElduQueAut}). The choice of this special
basis is motivated on the fact that it will turn to be an orthonormal
basis with respect to the norm in (\ref{eq:Norm-Ok}) and that through
a special bijective map between Okubo algebra and octonions the elements
of the basis $\left\{ e,\text{i}_{1},...,\text{i}_{7}\right\} $ here
defined will correspond to the octonionic one previously defined. 

Let us consider the quadratic form $n$ over Okubo algebra, given
by 
\begin{equation}
n\left(x\right)=\frac{1}{6}\text{Tr}\left(x^{2}\right),\label{eq:Norm-Ok}
\end{equation}
for every $x\in\mathcal{O}$. It is straightforward to see that this
\emph{norm} has signature $(8,0)$, is associative and composition
over the real Okubo algebra, i.e. 
\begin{align}
n\left(x*y\right) & =n\left(x\right)n\left(y\right),\\
\left\langle x*y,z\right\rangle  & =\left\langle x,y*z\right\rangle ,\label{eq:symmetric polar-1}
\end{align}
where $\left\langle \cdot,\cdot\right\rangle $ is the \emph{polar
form} given by 
\begin{equation}
\left\langle x,y\right\rangle =n\left(x+y\right)-n\left(x\right)-n\left(y\right).\label{eq:polar form}
\end{equation}
Therefore, Okubo algebra is a \emph{symmetric composition} algebra
\cite[Ch. VIII]{KMRT} and, thus, enjoying the notable relation 
\begin{equation}
x*\left(y*x\right)=\left(x*y\right)*x=n\left(x\right)y.\label{eq:symm-comp}
\end{equation}

For our purposes it will be of paramount importance to notice the
following \cite{OkMy80} 
\begin{prop}
\label{prop:division}The Okubo Algebra is a division algebra.
\end{prop}

\begin{proof}
Without any loss of generality, let us suppose that $d\neq0$ is a
left divisor of zero, i.e. $d*x=0$, then 
\[
n\left(d*x\right)=n\left(d\right)n\left(x\right)=0.
\]
But, since the algebra is symmetric composition algebra, for the (\ref{eq:symm-comp})
we also have

\begin{align}
\left(d*x\right)*d & =0=n\left(d\right)x,
\end{align}
and therefore $n\left(d\right)=0$, i.e. $\text{Tr}\left(d^{2}\right)=0$.
But, since the element $d$ is of the form 
\begin{equation}
d=\left(\begin{array}{ccc}
\xi_{1} & x_{1}+\text{i}y_{1} & x_{2}+\text{i}y_{2}\\
x_{1}-\text{i}y_{1} & \xi_{2} & x_{3}+\text{i}y_{3}\\
x_{2}-\text{i}y_{2} & x_{3}-\text{i}y_{3} & -\xi_{1}-\xi_{2}
\end{array}\right),
\end{equation}
where $x_{i},y_{i},\xi_{i}\in\mathbb{R}$, the norm $n\left(d\right)$
is given by 
\begin{equation}
n\left(d\right)=\frac{1}{3}\left(x_{1}^{2}+x_{2}^{2}+x_{3}^{2}+y_{1}^{2}+y_{2}^{2}+y_{3}^{2}+\xi_{1}^{2}+\xi_{2}^{2}+\xi_{1}\xi_{2}\right),\label{eq:norma-d}
\end{equation}
which yields that $\text{Tr}\left(d^{2}\right)\neq0$ in case of $\xi_{1},$$\xi_{2}\in\mathbb{R}$
and $\xi_{1},\xi_{2}\neq0$. 
\end{proof}
Unfortunately, since $\mathcal{O}$ is not a unital algebra, an element
$x$ does not have an inverse. This implies that, concerning its product,
the Okubo algebra is not a loop (as it was in the case of the octonions
that were a Moufang loop) but only a quasigroup. Nevertheless, considering
the existence of the idempotent $e$, and inspired by the identity
\[
x*\left(e*x\right)=\left(x*e\right)*x=n\left(x\right)e,
\]
we can define $\left(x\right)_{L}^{-1}=n\left(x\right)^{-1}\left(e*x\right)$
and $\left(x\right)_{R}^{-1}=n\left(x\right)^{-1}\left(x*e\right)$
so that, given a definite choice of the idempotent $e$, one has 
\[
\left(x\right)_{L}^{-1}*x=x*\left(x\right)_{R}^{-1}=e.
\]
As an implication of the previous argument we have the following 
\begin{prop}
\label{prop:SolutionLinearEq a*x=00003Db}An equation of the kind
\begin{equation}
a*x=b,\,\,\text{or}\,\,\,\,x*a=b,
\end{equation}
has a unique solution which is respectively given by 
\begin{equation}
x=\frac{1}{n\left(a\right)}b*a,\,\,\,\text{or}\,\,\,\,x=\frac{1}{n\left(a\right)}a*b,
\end{equation}
for every $a,b\in\mathcal{O}$, with $a\neq0$.
\end{prop}

\begin{proof}
Let us consider the equation $a*x=b$. Since $\mathcal{O}$ is a division
algebra and $a\neq0$ we can multiply by $a$ obtaining
\begin{equation}
\left(a*x\right)*a=b*a,
\end{equation}
but since $\left(a*x\right)*a=n\left(a\right)x$ and $n\left(a\right)\in\mathbb{R}$,
we then have $x=n\left(a\right)^{-1}b*a$. A similar argument is valid
for the case of $x*a=b$.
\end{proof}
Although the above proposition is straightforward, it has profound
geometrical implications, as it confirms the applicability of affine
and projective axioms to planes over the Okubo algebra. This topic
will be elaborated upon in subsequent sections.

\subsection{\label{subsec:Conjugation-and-the}Conjugation and the Trivolution}

In unital composition algebras, as noted earlier, there exists a canonical
involution, an order-two antihomomorphism known as \emph{conjugation}.
This can be defined using the orthogonal projection of the unit element
as
\begin{equation}
x\mapsto\overline{x}=\left\langle x,1\right\rangle 1-x.\label{eq:coniugazione}
\end{equation}
This canonical involution has the distinctive property of being an
antihomomorphism with respect to the product, i.e., $\overline{x\cdot y}=\overline{y}\cdot\overline{x},$
and the basic property with the norm of $x\cdot\overline{x}=n\left(x\right)1$.

For non-unital composition algebras, the previous definition isn't
applicable. However, if an idempotent element $e$ is present in the
algebra, one might be tempted to extend the previous definition 

\begin{equation}
x\mapsto\widetilde{x}=\left\langle x,e\right\rangle e-x,
\end{equation}
to investigate if similar properties remain valid.

In the case of a para-Hurwitz $p\mathbb{K}$ obtained from an Hurwitz
algebra $\left(\mathbb{K},\cdot,n\right)$ imposing the new product
$x\bullet y=\overline{x}\cdot\overline{y},$for every $x,y\in\mathbb{K}$
we have a special element, called para-unit, i.e. $1\in p\mathbb{K}$
such that $1\bullet x=\overline{x}.$ Thus, we might want to have
a look to the map $L_{1}$ given by left multiplication by the para-unit,
i.e.
\begin{align}
x & \mapsto L_{1}\left(x\right)=1\bullet x.
\end{align}
Clearly, the same arguments apply to $R_{1}\left(x\right)$ since
$x\bullet1=1\bullet x$. Indeed, we notice that $L_{1}^{2}\left(x\right)=x$
thus is an involution and, since 
\begin{align}
L_{1}\left(x\bullet y\right) & =1\bullet\left(x\bullet y\right)=y\cdot x,
\end{align}
and
\begin{align}
L_{1}\left(x\right)\bullet L_{1}\left(y\right) & =\overline{x}\bullet\overline{y}=x\cdot y,
\end{align}
the map is also an anti-homomorphism, i.e. $L_{1}\left(x\bullet y\right)=L_{1}\left(y\right)\bullet L_{1}\left(x\right)$.
Finally, since 
\begin{equation}
x\bullet L_{1}\left(x\right)=x\bullet1\bullet x=n\left(x\right)1.
\end{equation}
Thus, the order-two anti-homomorphism $x\mapsto\widetilde{x}=L_{1}\left(x\right)$
realises over the para-Hurwitz algebra $p\mathbb{K}$ all the main
features of the canonical involution or conjugation of the Hurwitz
algebra $\mathbb{K}$. 

Unfortunately, the situation within the Okubo algebra is less straightforward.
Indeed, if we consider the idempotent $e$ and define the map
\begin{equation}
x\mapsto\left\langle x,e\right\rangle e-x,
\end{equation}
it exhibits order two but is not neither a homomorphism or an antihomomorphism.
On the other hand, if we consider the maps 
\begin{align}
x & \longrightarrow L_{e}\left(x\right)=e*x,\label{eq:azione a sinistra}\\
x & \longrightarrow R_{e}\left(x\right)=x*e,\label{eq:azioni a destra}
\end{align}
 then we do have in both cases a nice relation with the norm, since
\begin{align}
x*L_{e}\left(x\right) & =n\left(x\right)e,\\
R_{e}\left(x\right)*x & =n\left(x\right)e,
\end{align}
Yet, even if $R_{e}\circ L_{e}=\text{id}$ holds true as in the para-Hurwitz
case, neither $L_{e}$ and $R_{e}$ are automorphism nor antiautomorphism.
On the other hand, if we generalise (\ref{eq:coniugazione}) with
the following map
\begin{equation}
x\mapsto\left\langle x,e\right\rangle e-x*e,
\end{equation}
 we have indeed a special automorphism, that we call $\tau$, which,
nevertheless, is not of order two but of order three. Therefore, while
it is not possible to have a involution over Okubo algebra that enjoys
the same properties of the conjugation of Hurwitz algebras, it is
possible to define something in a similar fashion such as an order-three
automorphism $\tau$, hereafter referred to as a\emph{ trivolution},
defined as 
\begin{equation}
x\longrightarrow\tau\left(x\right)=\left\langle x,e\right\rangle e-x*e,\label{eq:tau}
\end{equation}
or, equivalently, as
\begin{align}
x & \longrightarrow\tau\left(x\right)=L_{e}\left(x\right)^{2}=e*\left(e*x\right),\\
x & \longrightarrow\tau^{2}\left(x\right)=R_{e}\left(x\right)^{2}=\left(x*e\right)*e.
\end{align}
It is easy to see that the automorphism $\tau$ is of order $3$ since,
applying flexibility, we have $R_{e}^{2}\circ L_{e}^{2}=\text{id}$.
It is also worth noting the stunning analogy with the conjugation
expressed for unital composition algebra in (\ref{eq:coniugazione})
and at the same time the analogy with the one expressed for para-Hurwitz
algebras in (\ref{eq:azione a sinistra}). 

Even more interesting, the order-three automorphism $\tau$ is also
an order-three automorphism over the octonions $\mathbb{O}$. Indeed,
if we consider the base given in (\ref{eq:definizione i ottonioniche})
and set $e=\text{i}_{0}$, we can define $\tau$ as the linear map
given by
\begin{equation}
\begin{array}{cc}
\tau\left(\text{i}_{k}\right) & =\text{i}_{k},k=0,1,3,7\\
\tau\left(\text{i}_{2}\right) & =-\frac{1}{2}\left(\text{i}_{2}-\sqrt{3}\text{i}_{5}\right),\\
\tau\left(\text{i}_{5}\right) & =-\frac{1}{2}\left(\text{i}_{5}+\sqrt{3}\text{i}_{2}\right),\\
\tau\left(\text{i}_{4}\right) & =-\frac{1}{2}\left(\text{i}_{4}-\sqrt{3}\text{i}_{6}\right),\\
\tau\left(\text{i}_{6}\right) & =-\frac{1}{2}\left(\text{i}_{6}+\sqrt{3}\text{i}_{4}\right),
\end{array}\label{eq:Tau(Octonions)}
\end{equation}

Such definition extends to an order-three homomorphism over octonions
$\mathbb{O}$ once we consider $\left\{ \text{i}_{0}=1,\text{i}_{2},...,\text{i}_{7}\right\} $
a basis for this algebra. It is interesting to note that in the octonions
there are two Argan planes, generated by $\left\{ \text{i}_{2},\text{i}_{5}\right\} $
and $\left\{ \text{i}_{4},\text{i}_{6}\right\} $ on which the automorphism
$\tau$ acts as the cubic root of unity $\frac{1}{2}\left(1+\sqrt{3}\text{i}\right)$.
For completeness we give also the action of the inverse $\tau^{-1}$
over such basis, i.e. 
\begin{equation}
\begin{array}{cc}
\tau^{2}\left(\text{i}_{k}\right) & =\text{i}_{k},k=0,1,3,7\\
\tau^{2}\left(\text{i}_{2}\right) & =-\frac{1}{2}\left(\text{i}_{2}+\sqrt{3}\text{i}_{5}\right),\\
\tau^{2}\left(\text{i}_{5}\right) & =-\frac{1}{2}\left(\text{i}_{5}-\sqrt{3}\text{i}_{2}\right),\\
\tau^{2}\left(\text{i}_{4}\right) & =-\frac{1}{2}\left(\text{i}_{4}+\sqrt{3}\text{i}_{6}\right),\\
\tau^{2}\left(\text{i}_{6}\right) & =-\frac{1}{2}\left(\text{i}_{6}-\sqrt{3}\text{i}_{4}\right).
\end{array}\label{eq:Tau(Octonions)-1}
\end{equation}

\subsection{Okubo algebra, octonions and para-octonions}

An important feature of the Okubo algebra $\mathcal{O}$ is its interplay
with the algebra of octonions $\mathbb{O}$. Indeed, octonions and
the Okubo algebra are linked one another in such a way that we can
easily pass from one to the other simply changing the definition of
the bilinear product over the vector space of the algebra. Let us
consider the Kaplansky's trick we introduced earlier and let us define
a new product over the Okubo algebra $\mathcal{O}$ as
\begin{equation}
x\cdot y=\left(e*x\right)*\left(y*e\right),
\end{equation}
where $x,y\in\mathcal{O}$ and $e$ is an idempotent of $\mathcal{O}$.
Given that $e*e=e$ and $n\left(e\right)=1$, the element $e$ acts
as a left and right identity, i.e. 
\begin{align}
x\cdot e & =e*x*e=n\left(e\right)x=x,\\
e\cdot x & =e*x*e=n\left(e\right)x=x.
\end{align}
Moreover, since Okubo algebra is a composition algebra, the same norm
$n$ enjoys the following relation 
\begin{equation}
n\left(x\cdot y\right)=n\left(\left(e*x\right)*\left(y*e\right)\right)=n\left(x\right)n\left(y\right),
\end{equation}
which means that $\left(\mathcal{O},\cdot,n\right)$ is a unital composition
algebra of real dimension $8$. Since it is also a division algebra,
then it must be isomorphic to that of octonions $\mathbb{O}$ as noted
by Okubo himself \cite{Okubo 1978,Okubo 78c}. 

On the other hand, if we consider the order three automorphism of
the octonions in (\ref{eq:Tau(Octonions)}), the Okubo algebra is
then realised as a Petersson algebra from the octonions setting
\begin{align}
x*y & =\tau\left(\overline{x}\right)\cdot\tau^{2}\left(\overline{y}\right).
\end{align}
Note that (\ref{eq:tau}) is formulated assuming the knowledge of
the Okubic product. Reading the same maps as Okubic maps we then have
the notable relation, i.e.

\begin{align}
\overline{x} & =R_{e}^{3}\left(x\right)=\left(\left(x*e\right)*e\right)*e,\\
\tau\left(x\right) & =R_{e}^{4}\left(x\right)=\left(\left(\left(x*e\right)*e\right)*e\right)*e,
\end{align}
so that, in fact, the two maps are linked one another, i.e.

\begin{align}
\tau\left(x\right) & =\overline{x}*e\\
\overline{x} & =\tau\left(e*x\right).
\end{align}
While these maps are intertwined, it's important to highlight their
distinct impacts on the algebra's structure. While $\tau$ is an automorphism
for both Okubo algebra $\mathcal{O}$ and octonions $\mathbb{O}$,
$\overline{x}$ do not respect the algebrical structure of the Okubo
algebra $\mathcal{O}$, since it is not an automorphism nor an anti-automorphism
with respect to the Okubic product, while it is an anti-homomorphism
over octonions $\mathbb{O}$.

The scenario with para-octonions, $p\mathbb{O}$ is more straightforward.
By definition, para-octonions are obtainable from octonions $\mathbb{O}$
through
\begin{equation}
x\bullet y=\overline{x}\cdot\overline{y},
\end{equation}
while, on the other hand, octonions $\mathbb{O}$ are obtainable from
para-octonions $p\mathbb{O}$ through the aid of the para-unit $1\in p\mathbb{O}$,
such that
\begin{align}
x\cdot y & =\left(1\bullet x\right)\bullet\left(y\bullet1\right)\\
 & =\overline{x}\bullet\overline{y}=x\cdot y.
\end{align}
The new algebra $\left(p\mathbb{O},\cdot,n\right)$ is again an eight-dimensional
composition algebra which is also unital and division and thus, for
Hurwitz theorem, isomorphic to that of octonions $\mathbb{O}$. Moreover,
since $\tau\left(\overline{x}\right)=\overline{\tau\left(x\right)}$,
we also have that the Okubic algebra is obtainable from the para-Hurwitz
algebra with the introduction of a Petersson-like product, i.e. 
\begin{equation}
x*y=\tau\left(x\right)\bullet\tau^{2}\left(y\right).
\end{equation}
 We thus have that all algebras are obtainable one from the other
as summarized in Table \ref{tab:Oku-Para-Octo}. 
\begin{table}
\centering{}%
\begin{tabular}{|c|c|c|c|}
\hline 
Algebra & $\left(\mathcal{O},*\right)$ & $\left(p\mathbb{O},\bullet\right)$ & $\left(\mathbb{O},\cdot\right)$\tabularnewline
\hline 
\hline 
$x*y$ & $x*y$ & $\tau\left(x\right)\bullet\tau^{2}\left(y\right)$ & $\tau\left(\overline{x}\right)\cdot\tau^{2}\left(\overline{y}\right)$\tabularnewline
\hline 
$x\bullet y$ & $\tau^{2}\left(x\right)*\tau\left(y\right)$ & $x\bullet y$ & $\overline{x}\cdot\overline{y}$\tabularnewline
\hline 
$x\cdot y$ & $\left(e*x\right)*\left(y*e\right)$ & $\left(\boldsymbol{1}\bullet x\right)\bullet\left(y\bullet\boldsymbol{1}\right)$ & $x\cdot y$\tabularnewline
\hline 
\end{tabular}\caption{\label{tab:Oku-Para-Octo}In this table we see how to obtain the Okubic
product $*,$ the para-octonionic product $\bullet$ and the octonionic
product $\cdot$ from Okubo algebra $\left(\mathcal{O},*\right)$,
para-octonions $\left(p\mathbb{O},\bullet\right)$ and octonions $\left(\mathbb{O},\cdot\right)$
respectively}
\end{table}
 Nonetheless, it's vital to note that while transitioning from one
algebra to another is feasible, these algebras are not isomorphic.
For example, while the octonions $\mathbb{O}$ are alternative and
unital, para-octonions $p\mathbb{O}$ are nor alternative nor unital
but do have a para-unit. In contrast, the Okubo algebra $\mathcal{O}$
is non-alternative and only contains idempotent elements.

\section{\label{sec:Affine-and-projective}Affine and projective planes}

An \emph{incidence plane} $P^{2}$ is given by the triple $\left\{ \mathscr{P},\mathscr{L},\mathscr{R}\right\} $
where $\mathscr{P}$ is the set of points of the plane, $\mathscr{L}$
is the set of lines and $\mathscr{R}$ are the incidence relations
of poins and lines. The plane $P^{2}$ is called an \emph{affine plane}
if it satisfies the axioms of affine geometry, which means that the
relations $\mathscr{R}$ are such that: two points are joined by a
single line; any two non-parallel lines intersect in one point; finally,
for each line and each point there is a unique line which passes through
the point and it is parallel to the line. Instead, for $P^{2}$ to
be projective $\mathscr{R}$ must satisfy the property for which:
\begin{enumerate}
\item Any two distinct points are incident to a unique line.
\item Any two distinct lines are incident with a unique point. 
\item (\emph{not degenerate}) There exist four points such that no three
are incident one another.
\end{enumerate}
From definitions provided earlier, both affine and projective planes
can be defined in very abstract terms. However, in this section, our
focus is on the study of incidence planes defined in a natural way
over algebras: one aimed to generalize the construction of classical
affine and projective planes\cite{Compact Projective}. Our interest
lies in the definition wherein points of the affine plane are characterized
by two coordinates, where lines are linear functions of the coordinates
with respect to the a sum and a product that are those of the algebra
itself. To achieve a projective completion, it becomes necessary to
introduce an additional line at infinity and a few suitable relations.
This foundational approach to constructing affine and projective planes
proves effective -with minor yet meaningful modifications- for all
three 8-dimensional division composition algebras: octonions $\mathbb{O}$
, para-octonions $p\mathbb{O}$ and Okubo algebra $\mathcal{O}$.
Thus, we outline the framework for defining the affine and projective
planes for all three cases. However, in order to avoid repetitions
we will develop in all the details only for the Okubo case $\mathcal{O}$,
highlighting in Subsection \ref{subsec:Octonionic-and-para-octonionic}
the differences and variations needed in the other two cases.

\subsection{The Okubic affine plane and its completion}

A direct consequence of Proposition \ref{prop:SolutionLinearEq a*x=00003Db}
is the feasibility of defining affine geometry over the Okubo algebra
or, in other words, an Okubic affine plane $\mathscr{A}_{2}\left(\mathcal{O}\right)$
that satisfies all axioms of affine geometry. 
\begin{figure}
\centering{}\includegraphics[scale=0.38]{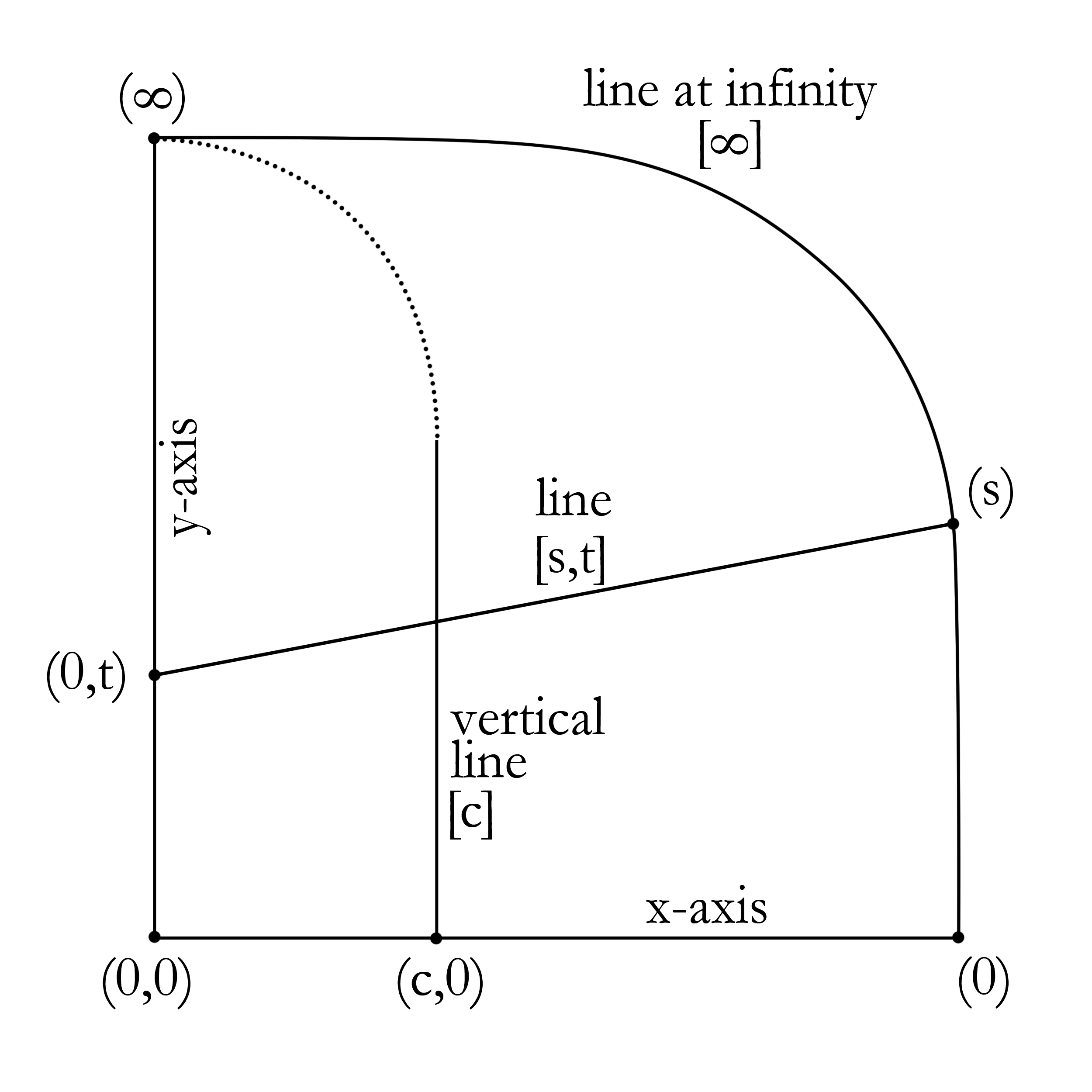}\caption{\label{fig:The affine plane}Representation of the completion of the
affine plane: $\left(0,0\right)$ represents the origin, $\left(0\right)$
the point at the infinity on the $x$-axis, $\left(s\right)$ is the
point at infinity of the line $\left[s,t\right]$ of slope $s$ while
$\left(\infty\right)$ is the point at the infinity on the $y$-axis
and of vertical lines $\left[c\right]$.}
\end{figure}
 Indeed, we identify a \emph{point} on the Okubic affine plane $\mathscr{A}_{2}\left(\mathcal{O}\right)$
by two coordinates $\left(x,y\right)$ with $x,y\in\mathcal{O}$,
while a \emph{line} of slope $s\in\mathcal{O}$ and offset $t\in\mathcal{O}$
is the set $\left[s,t\right]=\left\{ \left(x,s*x+t\right):x\in\mathcal{O}\right\} $.
Thus, the $x$ axis is represented by the line $\left[0,0\right]$.
On the other hand, \emph{vertical lines} are identified by $\left[c\right]$
which stands for the set $\left\{ c\right\} \times\mathcal{O}$. Here
$c\in\mathcal{O}$ represents the intersection with the $x$ axis,
thus $\left[0\right]$ denotes the $y$ axis. Finally, as for the
incidence rules we say that a point $\left(x,y\right)\in\mathscr{A}_{2}\left(\mathcal{O}\right)$
is \emph{incident }to a line $\left[s,t\right]\subset\mathscr{A}_{2}\left(\mathcal{O}\right)$
if belongs to such line, i.e. $\left(x,y\right)\in\left[s,t\right]$.

We now proceed to prove that the set of points, lines and incidence
relations previously defined forms an affine plane. 
\begin{thm}
The Okubic affine plane $\mathscr{A}_{2}\left(\mathcal{O}\right)$
with the previous incidence rules satisfies the axioms of affine geometry.
\end{thm}

\begin{proof}
First of all, we can straightforwardly see that given any two points
$\left(x_{1},y_{1}\right)$ and $\left(x_{2},y_{2}\right)$ there
is a unique line joining them. If $x_{1}=x_{2}=x$, the line is simply
$\left[x\right]$. On the other hand, if $x_{1}\neq x_{2}$, the line
is given by $\left[s,y_{1}-s*x_{1}\right]$, where $s$ is determined
by the linear equation
\[
s*\left(x_{1}-x_{2}\right)=\left(y_{1}-y_{2}\right),
\]
which has a unique solution given by
\begin{equation}
s=\frac{\left(x_{1}-x_{2}\right)*\left(y_{1}-y_{2}\right)}{n\left(x_{1}-x_{2}\right)}.\label{eq:SlopeOku}
\end{equation}
Similarly, for two lines $\left[s_{1},t_{1}\right]$ and $\left[s_{2},t_{2}\right]$
with distinct slopes $s_{1}\neq s_{2}$, a unique point of intersection
exists, i.e. $\left\{ \left(x,s_{1}*x+t_{1}\right)\right\} $ where
$x$ is 

\begin{equation}
x=\frac{\left(t_{2}-t_{1}\right)*\left(s_{1}-s_{2}\right)}{n\left(s_{1}-s_{2}\right)}.\label{eq:InterOku}
\end{equation}
If two lines have the same slope, they are disjoint. Two such lines
are called \emph{parallel}. Finally, for each line $\left[s,t\right]$
and each point $\left(x,y\right)$ there is a unique line, given by
i.e. $\left[s,y-s*x\right]$, which passes through $\left(x,y\right)$
and that is parallel to $\left[s,t\right]$. 
\end{proof}
The projective completion of the affine plane $\overline{\mathscr{A}_{2}}\left(\mathcal{O}\right)$
is obtained adding a line at infinity $\left[\infty\right]$, i.e.
\[
\left[\infty\right]=\left\{ \left(s\right):s\in\mathcal{O}\cup\left\{ \infty\right\} \right\} ,
\]
where $\left(s\right)$ identifies the end point at infinity of a
line with slope $s\in\mathcal{O}\cup\left\{ \infty\right\} $. Finally,
we define $\left(\infty\right)$ the point at infinity of $\left[\infty\right]$.
We now proceed to verify that the plane $\overline{\mathscr{A}_{2}}\left(\mathcal{O}\right)$
satisfies axioms of projective geometry: every two lines intersect
in a unique point; for every two points passes a unique line; there
are at least four points that form a non degenrate quadrangle. Indeed,
have the following 
\begin{thm}
The extended Okubic affine plane $\overline{\mathscr{A}_{2}}\left(\mathcal{O}\right)$
is a projective plane.
\end{thm}

\begin{proof}
First we need to show that for every two points of the extended plane
it still passes a unique line. This is straightforward since if the
points are of the affine plane the line was already determined; if
are both of them at infinity, i.e. $\left(s\right)$ and $\left(s'\right)$,
then such line is $\left[\infty\right]$; finally, if one is on the
affine plane $\left(x,y\right)$ and the other is at infinity $\left(s\right)$,
the line that joins them is $\left[s,y-s*x\right].$ On the other
hand, if two lines are not parallel the interstection was already
determined; if they are parallel lines, such as $\left[s,t_{1}\right]$
and $\left[s,t_{2}\right]$, they now intersect in the point $\left(s\right)$;
finally,while two vertical lines intersect in $\left(\infty\right)$.
The only thing that is left is to verify that it exists a non-degenerate
quadrangle where no three points are collinear which in this case
can be found easily, e.g. the quadrangle $\diamondsuit=\left\{ \left(0,0\right),\left(e,e\right),\left(0\right),\left(\infty\right)\right\} \subset\overline{\mathscr{A}_{2}}\left(\mathcal{O}\right)$
is such that no three elements are incident to the same line. Indeed,
the lines that join those points are $\left[0,0\right],\left[0\right],\left[\infty\right],\left[0,e\right],\left[e,0\right]$
and $\left[e\right]$ and none of those contains three elements of
$\diamondsuit$ .
\end{proof}
\begin{rem}
As in the standard projective plane over a field, we would like to
point out to the reader the existence of a fundamental triangle also
in the extended Okubic affine plane. More precisely, the entirety
of the affine plane is encompassed by a triangle given by three special
points: the \emph{origin} $\left(0,0\right)$; the \emph{0-point at
infinity}, i.e. the point $\left(0\right)$ obtained prolonging the
$x$ axis to infinity; finally, the \emph{$\infty$-point at infinity},
i.e. the point $\left(\infty\right)$ obtained prolonging the $y$
axis to infinity. We designate $\triangle$ the set made by those
three points, i.e. $\triangle=\left\{ \left(0,0\right),\left(0\right),\left(\infty\right)\right\} .$
\end{rem}

\subsection{\label{sec:The-Okubic-projective}The Okubic projective plane}

We will now define directly the projective plane $\mathcal{O}P^{2}$
and subsequently illustrate its correspondence with the completion
of the affine plane $\overline{\mathscr{A}_{2}}\left(\mathcal{O}\right)$.
Historically, numerous tricks were used for defining projective planes
over non-associative algebras such as octonions. Here we will use
a variation of the one proposed by H. Salzmann \cite{Compact Projective}
which is based on what he calls ``Veronese coordinates''. Let $V$
be the 27-dimensional real vector space $V\cong\mathbb{\mathcal{O}}^{3}\times\mathbb{R}^{3}$
, with elements of the form 
\[
\left(x_{\nu};\lambda_{\nu}\right)_{\nu}=\left(x_{1},x_{2},x_{3};\lambda_{1},\lambda_{2},\lambda_{3}\right),
\]
where $x_{\nu}\in\mathcal{O}$, $\lambda_{\nu}\in\mathbb{R}$ and
$\nu=1,2,3$. We then define the Veronese vectors to be those $w\in V$
that satisfy the following \emph{Veronese conditions,}

\begin{align}
\lambda_{1}x_{1} & =x_{2}*x_{3},\,\,\lambda_{2}x_{2}=x_{3}*x_{1},\,\,\lambda_{3}x_{3}=x_{1}*x_{2},\label{eq:Okubo Ver-1}\\
n\left(x_{1}\right) & =\lambda_{2}\lambda_{3},\,n\left(x_{2}\right)=\lambda_{3}\lambda_{1},n\left(x_{3}\right)=\lambda_{1}\lambda_{2}.\label{eq:Okubo Ver-2}
\end{align}
It is straightforward to see that if $w=\left(x_{\nu};\lambda_{\nu}\right)_{\nu}$
is a Veronese vector then also $\mu w=\mu\left(x_{\nu};\lambda_{\nu}\right)_{\nu}$
is Veronese for every $\mu\in\mathbb{R}$. The set of Veronese vectors
is therefore a subset that we will call $H$ and for ever $w$ that
is Veronese we indicate as $\mathbb{R}w\subset H$ the class of real
multiples of $w$. The \emph{Okubic projective plane} $\mathcal{O}P^{2}$
is then the geometry having the 1-dimensional subspaces $\mathbb{R}w$
as \emph{points}, i.e. 
\begin{equation}
\mathscr{P}_{\mathcal{O}}=\left\{ \mathbb{R}w:w\in H\smallsetminus\left\{ 0\right\} \right\} .\label{eq:Projective plane}
\end{equation}
The set of \emph{lines $\mathscr{L}_{\mathcal{O}}$ }is formed by
subspaces $\ell_{w}$ in the projective plane $\mathcal{O}P^{2}$
that are orthogonal to a Veronese vector $w\in H$, i.e. 
\begin{equation}
\ell_{w}=w^{\bot}=\left\{ z\in H:\beta\left(z,w\right)=0\right\} ,\label{eq:Projective line}
\end{equation}
where the bilinear form $\beta$ is the extension to $V$ of the polarisation
of the Okubic norm. More specifically, defining $\left\langle x,y\right\rangle =n\left(x+y\right)-n\left(x\right)-n\left(y\right)$,
for any two Okubic elements $x,y\in\mathbb{\mathcal{O}}$, the bilinear
form $\beta$ is given by

\begin{equation}
\beta\left(v,w\right)=\stackrel[\nu=1]{3}{\sum}\left(\left\langle x_{\nu},y_{\nu}\right\rangle +\lambda_{\nu}\eta_{\nu}\right),\label{eq:beta bilinear}
\end{equation}
where $v=\left(x_{\nu};\lambda_{\nu}\right)_{\nu}$ and $w=\left(y_{\nu};\eta_{\nu}\right)_{\nu}$
are Veronese vectors in $H\subset V$.

Finally, the incidence relations are again given by inclusion $\subseteq$,
i.e. we say that a point $\mathbb{R}w\in\mathcal{O}P^{2}$ is \emph{incident}
to the line $\ell_{v}$ iff $\mathbb{R}w\in v^{\bot}$, i.e. $\beta\left(w,v\right)=0$.
\begin{rem}
Since all real multiple of a Veronese vectors $v=\left(x_{\nu};\lambda_{\nu}\right)_{\nu}$
identify the same point on the projective line, we will usually take
as representative of the class the one such that $\lambda_{1}+\lambda_{2}+\lambda_{3}=1$,
such that an alternative definition of the set of points in (\ref{eq:Projective plane})
could be 
\begin{equation}
\mathcal{O}P^{2}=\left\{ \left(x_{\nu};\lambda_{\nu}\right)_{\nu}\in H\smallsetminus\left\{ 0\right\} ,\lambda_{1}+\lambda_{2}+\lambda_{3}=1\right\} .
\end{equation}
\end{rem}

\medskip{}

\begin{rem}
It is also worth noting how the norm $n$, defined over the symmetric
composition algebra $\mathbb{\mathcal{O}}$, is intertwined with the
geometry of the plane. This relationship becomes evident when considering
the quadratic form of the bilinear symmetric form $\beta$, i.e.

\begin{equation}
q\left(v\right)\coloneqq\frac{1}{2}\beta\left(v,v\right)=n\left(x_{1}\right)+n\left(x_{2}\right)+n\left(x_{3}\right)+\frac{1}{2}\left(\lambda_{1}^{2}+\lambda_{2}^{2}+\lambda_{3}^{2}\right),\label{eq:Norm element projective}
\end{equation}
where $v=\left(x_{\nu};\lambda_{\nu}\right)_{\nu}$.
\end{rem}

In the next section we will show that the triple $\mathcal{O}P^{2}=\left\{ \mathscr{P}_{\mathcal{O}},\mathscr{L}_{\mathcal{O}},\subseteq\right\} $
is indeed a projective plane and, even more, is equivalent to the
completion of the affine plane $\overline{\mathscr{A}_{2}}\left(\mathcal{O}\right)$.

\subsection{\label{sec:Correspondence-between-affine}Correspondence between
affine and projective plane}

In establishing a one-to-one correspondence between the completion
of the affine plane $\overline{\mathscr{A}_{2}}\left(\mathcal{O}\right)$
and the projective plane $\mathcal{O}P^{2}$, we must ensure that
such a correspondence maintains the incidence relations. Specifically,
a point incident to a line in $\overline{\mathscr{A}_{2}}\left(\mathcal{O}\right)$
should map to a point in $\mathcal{O}P^{2}$ that is incident to the
image of the original line. As demonstrated in \cite{Corr-OkuboSpin},
the map which sends points and lines from $\overline{\mathscr{A}_{2}}\left(\mathcal{O}\right)$
to $\mathcal{O}P^{2}$ defined by

\begin{equation}
\begin{array}{ccc}
\left(x,y\right) & \mapsto & \mathbb{R}\left(x,y,x*y;n\left(y\right),n\left(x\right),1\right),\\
\left(x\right) & \mapsto & \mathbb{R}\left(0,0,x;n\left(x\right),1,0\right),\\
\left(\infty\right) & \mapsto & \mathbb{R}\left(0,0,0;1,0,0\right),\\
\left[s,t\right] & \mapsto & \left(t*s,-t,-s;1,n\left(s\right),n\left(t\right)\right)^{\bot},\\
\left[c\right] & \mapsto & \left(-c,0,0;0,1,n\left(c\right)\right)^{\bot},\\
\left[\infty\right] & \mapsto & \left(0,0,0;0,0,1\right)^{\bot},
\end{array}\label{eq:correspondence}
\end{equation}
is indeed well-defined and keeps the incidence relations. 
\begin{lem}
The aforementioned correspondence (\ref{eq:correspondence}) is well-defined
and is a one-to-one correspondence between points and lines of the
affine plane $\overline{\mathscr{A}_{2}}\left(\mathcal{O}\right)$
and points and lines of the projective plane $\mathcal{O}P^{2}$
\end{lem}

\begin{proof}
In fact, this is just a trivial check that relies on the Veronese
conditions and $\mathcal{O}$ being a symmetric composition algebra
for which just (\ref{eq:comp(Def)}) and (\ref{eq:x*y*x=00003Dn(x)y})
has to be used. For example, let $\left(x,y\right)$ be a point of
the affine plane, then the vector $\left(x,y,x*y;n\left(y\right),n\left(x\right),1\right)$
is a Veronese vector since a direct check of (\ref{eq:Okubo Ver-1})
and (\ref{eq:Okubo Ver-2}) yields to
\begin{equation}
\begin{array}{ccccc}
n\left(y\right)x=y*\left(x*y\right), &  & n\left(x\right)y=\left(x*y\right)*x,\,\, &  & x*y=x*y,\\
n\left(x\right)=n\left(x\right), &  & n\left(y\right)=n\left(y\right), &  & n\left(x*y\right)=n\left(x\right)n\left(y\right),
\end{array}
\end{equation}
 that are either identically true or obtainable from the fact that
Okubo algebra is a composition algebra, i.e. $n\left(x*y\right)=n\left(x\right)n\left(y\right)$,
or from the symmetric composition identity, i.e. $n\left(x\right)y=\left(x*y\right)*x.$
On the other hand, for any Veronese vector $v=\left(x_{\nu};\lambda_{\nu}\right)_{\nu}$
with $\lambda_{3}\neq0$ we have that subspace $\mathbb{R}v$ is the
same of 
\begin{equation}
\mathbb{R}v=\mathbb{R}\left(x,y,x*y;n\left(y\right),n\left(x\right),1\right),
\end{equation}
where $x=\lambda_{3}^{-1}x_{1}$ and $y=\lambda_{3}^{-1}x_{2}$ which
is again a Veronese vector. The check with a generic line proceeds
on the same way, but it might be interesting to explicitly check that
\begin{equation}
\left[\infty\right]\longrightarrow\left(0,0,0;0,0,1\right)^{\bot},
\end{equation}
is indeed a line. First of all, we need to find the Veronese vectors
orthogonal to $\left(0,0,0;0,0,1\right)$. These are vectors with
$\lambda_{3}=0$ and, therefore, with $n\left(x_{1}\right)=n\left(x_{2}\right)=0$,
and thus with $x_{1}=x_{2}=0$. Then, elements orthogonal to $\left(0,0,0;0,0,1\right)$
might take only two forms depending on $x_{3}$ being $0$ or $x_{3}\neq0$,
i.e.
\begin{equation}
\left(0,0,0;0,0,1\right)^{\bot}=\left\{ \mathbb{R}\left(0,0,x;n\left(x\right),1,0\right)\right\} \cup\left\{ \mathbb{R}\left(0,0,0;1,0,0\right)\right\} ,
\end{equation}
where $x\in\mathcal{O}$. In fact, these are in a trivial way elements
of an Okubic line $\mathcal{O}\cup\left\{ \infty\right\} $as one-point
compactification of Okubo algebra.
\end{proof}
While that (\ref{eq:correspondence}) is well defined and is a one-to-one
correspondence between $\overline{\mathscr{A}_{2}}\left(\mathcal{O}\right)$
and $\mathcal{O}P^{2}$ was a trivial check, the proof that incidence
rules are preserved by (\ref{eq:correspondence}) is a little bit
more involved and for this reason we will show it in a complete form
with the following
\begin{lem}
The correspondence in (\ref{eq:correspondence}) preserves the incidence
relations between $\overline{\mathscr{A}_{2}}\left(\mathcal{O}\right)$
and $\mathcal{O}P^{2}$.
\end{lem}

\begin{proof}
We need to show that the image of a point $\left(x,y\right)$ incident
to the line $\left[s,t\right]$ is mapped by (\ref{eq:correspondence})
into a point of the projective plane, i.e. $\mathbb{R}\left(x,y,x*y;n\left(y\right),n\left(x\right),1\right)$,
that is incident to the image of $\left[s,t\right]$, i.e. is incident
to $\left(t*s,-t,-s;1,n\left(s\right),n\left(t\right)\right)^{\bot}$.
By definition of incidence on the projective plane and of (\ref{eq:Projective line}),
the image of $\left(x,y\right)$ is incident to the image of $\left[s,t\right]$
if and only if the following condition is satisfied 
\begin{equation}
\left\langle t*s,x\right\rangle -\left\langle t,y\right\rangle -\left\langle s,x*y\right\rangle +n\left(y\right)+n\left(s\right)n\left(x\right)+n\left(t\right)=0.\label{eq:eqretta}
\end{equation}
Noting that 
\begin{equation}
\left\langle s*x,t-y\right\rangle =n\left(s*x+t-y\right)-n\left(s*x\right)-n\left(t-y\right),
\end{equation}
and since (\ref{eq:associativityNorm}), we then have
\begin{equation}
\begin{array}{cc}
\left\langle s*x,t-y\right\rangle  & =\left\langle s*x,t\right\rangle -\left\langle s*x,y\right\rangle \\
 & =\left\langle t,s*x\right\rangle -\left\langle s,x*y\right\rangle \\
 & =\left\langle t*s,x\right\rangle -\left\langle s,x*y\right\rangle ,
\end{array}
\end{equation}
and, therefore, 
\begin{equation}
\left\langle t*s,x\right\rangle -\left\langle s,x*y\right\rangle =n\left(s*x+t-y\right)-n\left(s*x\right)-n\left(t-y\right).
\end{equation}
 Inserting the latter into (\ref{eq:eqretta}) and noting that $n\left(s\right)n\left(x\right)=n\left(s*x\right)$,
then (\ref{eq:eqretta}) is equivalent to 
\begin{align}
n\left(s*x+t-y\right) & =0.
\end{align}
Since Okubo algebra is a division composition algebra, and the only
element of zero norm is zero, then (\ref{eq:eqretta}) is satisfied
iff $s*x+t-y=0$, that is $\left(x,y\right)\in\left[s,t\right]$.
The cases for the incidence of $\left(s\right)$ and $\left(\infty\right)$
with $\left[\infty\right]$ can be proved in the same way.
\end{proof}
Once is shown that the map (\ref{eq:correspondence}) gives a one
to one correspondence that keeps incidence relation we thus have the
following
\begin{thm}
\label{cor:The-Okubic-projective}The Okubic plane given by the triple
$\mathcal{O}P^{2}=\left\{ \mathscr{P}_{\mathcal{O}},\mathscr{L}_{\mathcal{O}},\subseteq\right\} $
is a projective plane and is isomorphic to the completion of the affine
plane $\overline{\mathscr{A}_{2}}\left(\mathcal{O}\right)$.
\end{thm}

\subsection{\label{subsec:Octonionic-and-para-octonionic}Octonionic and para-octonionic
planes}

Previous definitions pertaining to the Okubic plane can be generalized
to the para-octonionic case and, with minor variations, to the octonionic
case as referenced in \cite{corr Notes Octo,Compact Projective}.
Indeed, in the context of the paraoctonionic case, there is no need
to alter the definitions formally, provided we replace the Okubonic
product $*$ with the para-octonionic product $\bullet$. In a manner
analogous to the Okubic case, a point of the para-octonionic affine
plane $\mathscr{A}_{2}\left(p\mathbb{O}\right)$ is given by a pair
of elements $\left(x,y\right)$ with $x$,$y\in\left\{ p\mathbb{O}\right\} $,
while a \emph{line} of slope $s\in p\mathbb{O}$ and offset $t\in p\mathbb{O}$
is the set $\left[s,t\right]=\left\{ \left(x,s\bullet x+t\right):x\in p\mathbb{O}\right\} $
and, of course, we say that a point $\left(x,y\right)\in\mathscr{A}_{2}\left(p\mathbb{O}\right)$
is \emph{incident }to a line $\left[s,t\right]\subset\mathscr{A}_{2}\left(p\mathbb{O}\right)$
if belongs to such line, i.e. $\left(x,y\right)\in\left[s,t\right]$.
The octonionic case $\mathbb{O}$ with the product $\cdot$ follows
the same definitions. 

For the affine plane, distinctions primarily manifest in the octonionic
equations that describe the slope $s$ of the line passing through
two points of the plane and coordinate $x$ of the intersection of
two generic lines as found in (\ref{eq:SlopeOku}) and (\ref{eq:InterOku}).
In the para-octonionic scenario, the expressions remain as
\begin{equation}
s=\frac{\left(x_{1}-x_{2}\right)\bullet\left(y_{1}-y_{2}\right)}{n\left(x_{1}-x_{2}\right)},\,\,\,x=\frac{\left(t_{2}-t_{1}\right)\bullet\left(s_{1}-s_{2}\right)}{n\left(s_{1}-s_{2}\right)}.
\end{equation}
However, the octonionic variant introduces a slight modification due
to the unique properties of octonions as a unital composition algebra.
Given that $x^{-1}=\overline{x}/n\left(x\right)$, the equations transform
to
\[
s=\frac{\left(y_{1}-y_{2}\right)\cdot\overline{\left(x_{1}-x_{2}\right)}}{n\left(x_{1}-x_{2}\right)},\,\,\,x=\frac{\overline{\left(s_{1}-s_{2}\right)}\cdot\left(t_{2}-t_{1}\right)}{n\left(s_{1}-s_{2}\right)}.
\]
Similar modifications are observed in the projective planes' definitions,
as seen in (\ref{eq:Okubo Ver-1}) and (\ref{eq:Okubo Ver-2}). For
the para-octonionic case, given a vector $\left(x_{1},x_{2},x_{3};\lambda_{1},\lambda_{2},\lambda_{3}\right)\in p\mathbb{O}^{3}\times\mathbb{R}^{3}$,
one isolate the subset of Veronese vectors satisfying the following
conditions

\begin{align}
\lambda_{1}x_{1} & =x_{2}\bullet x_{3},\,\,\lambda_{2}x_{2}=x_{3}\bullet x_{1},\,\,\lambda_{3}x_{3}=x_{1}\bullet x_{2},\label{eq:Ver-1-paraOct-1}\\
n\left(x_{1}\right) & =\lambda_{2}\lambda_{3},\,n\left(x_{2}\right)=\lambda_{3}\lambda_{1},n\left(x_{3}\right)=\lambda_{1}\lambda_{2},\label{eq:Ver-2-paraOct-1}
\end{align}
which closely resemble the Okubic conditions. In contrast, the octonionic
variant relies on the Veronese conditions applicable to all Hurwitz
algebras, i.e.,

\begin{align}
\lambda_{1}\overline{x_{1}} & =x_{2}\cdot x_{3},\,\,\lambda_{2}\overline{x_{2}}=x_{3}\cdot x_{1},\,\,\lambda_{3}\overline{x_{3}}=x_{1}\cdot x_{2},\label{eq:Ver-1-Oct}\\
n\left(x_{1}\right) & =\lambda_{2}\lambda_{3},\,n\left(x_{2}\right)=\lambda_{3}\lambda_{1},n\left(x_{3}\right)=\lambda_{1}\lambda_{2}.\label{eq:Ver-2-Oct}
\end{align}

To conclude, these differences in the Veronese conditions correspond
to varied formulations of the one-to-one relationship between the
affine and projective planes. Within this relationship, the para-octonions
retain the formal mapping in (\ref{sec:Correspondence-between-affine}),
and more specifically one still has

\begin{align}
\left(x,y\right) & \mapsto\mathbb{R}\left(x,y,x\bullet y;n\left(y\right),n\left(x\right),1\right),\\
\left[s,t\right] & \mapsto\left(t\bullet s,-t,-s;1,n\left(s\right),n\left(t\right)\right)^{\bot},
\end{align}
while for the octonionic plane one has to modify them as follow
\begin{align}
\left(x,y\right) & \mapsto\mathbb{R}\left(x,y,\overline{y}\cdot x;n\left(y\right),n\left(x\right),1\right),\\
\left[s,t\right] & \mapsto\left(\overline{s}\cdot t,-t,-s;1,n\left(s\right),n\left(t\right)\right)^{\bot}.
\end{align}
Interestingly, as inferred from the above equations, the Veronese
conditions for para-octonions are simpler than those for octonions,
as we do not need to use the conjugation. This leads to an intriguing
observation: defining the Cayley plane appears more intuitive using
para-octonions than octonions.

\section{Collineations on the plane}

In this section we study the collineations of the Okubic affine and
projective plane. We start presenting explicit forms of elations,
more specifically translations and shears, and of the triality collineation
(see below). The direct study of the motion group is important since
it might be an alternative way in proving the isomorphism of the Okubic
plane with the Cayley plane. Indeed, it is well known that any 16-dimensional
compact plane with a collineation group of dimension greater than
40 it is isomorphic to the Cayley plane $\mathbb{O}P^{2}$ (see \cite[Chap. 8]{Compact Projective}).
In fact, this is not needed since we will write an explicit isomorphism
between the Okubic plane $\mathcal{O}P^{2}$, the paraoctonionic plane
$p\mathbb{O}P^{2}$ and the octonionic plane $\mathbb{O}P^{2}$ in
the next section. As result, the collineation groups of the three
planes coincide and is the exceptional Lie group $\text{E}_{6(-26)}$.
Nevertheless, it is noteworthy that a variation in the foundational
algebra defining the plane, despite preserving the overall collineation
group, alters the algebraic description of the collineations. Consequently,
in the Okubic realization of the 16-dimensional Moufang plane, the
reflection $\begin{array}{c}
\left(x,y\right)\longrightarrow\left(y,x\right)\end{array}$, is not a collineation, whereas it is in its octonionic realisation.

\subsection{\label{sec:Collineations}Collineations}

A \emph{collineation} is a bijection $\varphi$ of the set of points
of the plane onto itself, such that lines map to lines. Since the
identity map is a collineation, the inverse $\varphi^{-1}$ and the
composition $\varphi\circ\varphi'$ are collineations if $\varphi,\varphi'$
are both collineations, then the set of collineations is in fact a
group under composition that we will denote as $\text{Aut}\left(\mathcal{O}P^{2}\right)$.
A notable characteristic of collineations is that they keep incidence
relations of both affine and projective planes. Indeed, given two
points $p_{1}$ and $p_{2}$, there is only one line passing through
them and, clearly, the image of such line is the only one that passes
through $p_{1}^{\varphi}$ and $p_{2}^{\varphi}$ (where we used the
classical notation $p^{\varphi}$ and $\ell^{\varphi}$ to indicate
the image of the point $p$ and the line $\ell$ through the collineation
$\varphi$). We thus have the following 
\begin{prop}
Collineations of the affine plane send parallel lines into parallel
lines. 
\end{prop}

As a consequence of the previous proposition we also have the following
\begin{cor}
Any affine collineation can be extended uniquely as a projective collineation.
\end{cor}

\begin{proof}
Since an affine collineation sends parallel lines $\left[s,t\right]$
into parallel lines $\left[s,t\right]^{\varphi}$, so that in fact
we have that all lines with slope $s$ go in lines with slope $s^{\varphi}$.
Clearly, we can extend the affine collineation to the projective plane
if and only if we set that parallels lines go to the same point at
infinity, i.e. setting $\left(s\right)^{\varphi}=\left(s^{\varphi}\right)$. 
\end{proof}
A set of collineations $\triangle$ is called \emph{transitive} on
a set $M$ if for every $x,y\in M$ it exists a collineation $\varphi$
such that $x^{\varphi}=y$. On the other hand, a set of collineations
$\triangle$ is called \emph{doubly transitive} if for any quadruple
of points $x,y,z,w\in M$, it exists a collineation $\varphi$ such
that $x^{\varphi}=y$ and $z^{\varphi}=w$.

\subsection{Axial collineations }

Given a collineation $\varphi$ we say that $\varphi$ is \emph{axial}
if fixes every point of a line $\ell$. In this case, the line $\ell$
is called an \emph{axis} of $\varphi$. On the other hand we say that
$\varphi$ is \emph{central }if fixes every line passing through a
point $p$, which in this case it is called a \emph{center} of $\varphi$. 

It is known from a general setting of projective geometry that any
collineation of a projective plane is axial if and only if is central
(see \cite[Thm. 4.9]{HP}). Moreover, it is easy to see that an axial
collineation that has two centers or two axis is the identity. Indeed,
let us suppose that $\varphi$ has two center $p$ and $q$, then
any other point $r$ outside the line joining $p$ and $q$ would
be fixed since $r$ is given as intersection of two fixed lines, one
passing through $p$ and the other through $q$. On the other hand,
we could just replicate the argument for the point $r$ with $p$
and determine that the collineation must fix also the line joining
$p$ and $q$.

Given a point $p$ and a line $\ell$, we denote an axial collineation
with center $p$ and line $\ell$ as $\varphi_{\left[p,\ell\right]}$
and the group of such collineations as $\Gamma_{\left[p,\ell\right]}$.
It is then easy to verify what is known as the\emph{ conjugation formula},
i.e. 
\begin{lem}
\emph{\label{lem:(Conjugation-formula-)}(Conjugation formula \cite[Lemma. 4.11]{HP})}
For every collineation $\varphi$ the group $\Gamma_{\left[p^{\varphi},\ell^{\varphi}\right]}$
of collineations with center $p^{\varphi}$ and axis $\ell^{\varphi}$
is just the conjugate of $\Gamma_{\left[p,\ell\right]}$ through $\varphi$
\begin{equation}
\varphi^{-1}\circ\Gamma_{\left[p,\ell\right]}\circ\varphi=\Gamma_{\left[p^{\varphi},\ell^{\varphi}\right]}.
\end{equation}
\end{lem}

Moreover, an axial collineation $\varphi_{\left[p,\ell\right]}$ that
fixes a point $q$ outside $p\cup\ell$ is the identity. Indeed if
$q$ is fixed by $\varphi_{\left[p,\ell\right]}$ then joining the
points of $\ell$ with $q$ we would see that also $q$ is a center. 

Since axial collineations fix only a point called center and a line
called axis, they can be easily divided in two classes: 
\begin{enumerate}
\item those for which the center $p$ is incident to the axis $\ell$ and
that are called \emph{elations}; 
\item those collineations for which the center $p$ is not incident to the
axis $\ell$ and that are called \emph{homologies}.
\end{enumerate}
Finally, a last theorem it is worth reviewing, since it is a standard
argument that we will use. 
\begin{lem}
\emph{(see \cite[sec. 23.9]{Compact Projective})} \label{lem:Transitive}Suppose
that $\triangle$ is a set of collineations of center $p$ and axis
$\ell$, and let $m$ be a line through $p$ with $m\neq\ell$. If
$\triangle$ is transitive on the set of points of $m$ that are not
incident with the center or the axis, i.e., $m\smallsetminus\left\{ p,m\wedge\ell\right\} $
then $\triangle$ is the group of all collineations with center $p$
and axis $\ell$, i.e. $\Gamma_{\left[p,\ell\right]}$. 
\end{lem}

\subsection{\label{sec:Elations}Elations }

We now focus on a special class of axial collineations called \emph{elations},
i.e. that are those collineations in which the center is incident
to the axis.
\begin{cor}
Collineations of $\overline{\mathscr{A}_{2}}\left(\mathcal{O}\right)$
that have the line at infinity $\left[\infty\right]$ as axis and
center incident with the axis are precisely the translations 
\begin{equation}
\begin{array}{c}
\tau_{a,b}:\left(x,y\right)\longrightarrow\left(x+a,y+b\right),\\
\tau_{a,b}\mid_{\left[\infty\right]}=id.
\end{array}
\end{equation}
\end{cor}

\begin{proof}
First of all, we show that this are collineations that have axis $\left[\infty\right]$
and center incident to the axis. Indeed, given a line $\left[s,t\right]$
or $\left[c\right]$, then its image through $\tau_{a,b}$ is another
line given by 
\begin{align}
\left[s,t\right]^{\tau_{a,b}} & =\left[s,t-s*a+b\right],\label{eq:t-s*p+q}\\
\left[c\right]^{\tau_{a,b}} & =\left[c+a\right].
\end{align}
Since this are collineations of the affine plane, they do extend in
a unique way as collineations on the projective plane and since the
slope $\left(s\right)$ is unchanged, then the line at infinity is
the axis of the collineation. Moreover, let us now consider (\ref{eq:t-s*p+q}).
Clearly if $a\neq0$ there is a unique slope $s$, namely $s=n\left(s\right)^{-1}\left(a*b\right)$,
such that $\left[s,t\right]^{\tau_{a,b}}=\left[s,t\right]$ for every
$t\in\mathcal{O}$. But the set $\left\{ \left[s,t\right]:t\in\mathcal{O}\right\} $
is exactly the set of parallel lines that pass through the point $\left(s\right)$,
i.e. $\left(s\right)$ is a center of the collineation and is incident
to the axis $\left[\infty\right]$. The same reasoning can be applied
when $a=0$, since in that case all vertical lines $\left[c\right]$
with $c\in\mathcal{O}$ would be invariant and the center of the collineation
would be $\left(\infty\right)$. 

We now need to demonstrate that all elations with axis $\left[\infty\right]$
and center $\left(p\right)$ are of the form of $\tau_{a,b}$. First
of all since $\left[\infty\right]$ is the axis, which means that
the collineation fixes pointwise the line $\left[\infty\right]$,
then the image of a line of slope $\left(p\right)$ will be a line
of the same slope $\left(p\right)$. Now let be $q_{1},q_{2}$ any
two points in the affine plane incident to a line of slope $p$. Then
it exists a traslation of the form $\tau_{a,b}$ that sends $q_{1}$
in $q_{2}$ . The group of translations $\tau_{a,b}$ is thus transitive
on the line $M$ joining $q_{1}$ and $q_{2}$ which has slope $\left(p\right)$.
This means that the group of translations is that of all collineation
with center $\left(p\right)$ and axis $\left[\infty\right]$ by Lemma
\ref{lem:Transitive}.
\end{proof}
We now focus on elations that have vertical axis $\left[0\right]$
and center in $\left(\infty\right)$ which they too enjoy an easy
and elegant characterization.
\begin{thm}
Collineations of $\overline{\mathscr{A}_{2}}\left(\mathcal{O}\right)$
that have the vertical axis $\left[0\right]$ as axis and center in
$\left(\infty\right)$ are precisely the shears 
\begin{equation}
\begin{array}{cc}
\sigma_{a}: & \left(x,y\right)\longrightarrow\left(x,y+ax\right),\\
 & \left(s\right)\longrightarrow\left(s+a\right),\\
 & \left(\infty\right)\longrightarrow\left(\infty\right).
\end{array}\label{eq:shears}
\end{equation}
\end{thm}

\begin{proof}
First of all we show that this are collineations that have axis $\left[0\right]$
and center in $\left(\infty\right)$. Indeed, given a line $\left[s,t\right]$
or $\left[c\right]$, then its image through $\sigma_{a}$ is another
line given by 
\begin{align}
\left[s,t\right]^{\sigma_{a}} & =\left[s+a,t\right],\label{eq:t-s*p+q-1}\\
\left[c\right]^{\sigma_{a}} & =\left[c\right],
\end{align}
so that $\sigma_{a}$ are indeed collineations. Since all lines of
the form $\left[c\right]$ are invariant, therefore the point $\left(\infty\right)$
that joins them is the center of all $\sigma_{a}$ . On the other
hand, looking at (\ref{eq:shears}) it is evident that all points
of the for $\left(0,t\right)$ are fixed by all $\sigma_{a}$ and
thus $\left[0\right]$ is the axis. Since $\left(\infty\right)\in\left[0\right]$,
then $\sigma_{a}$ are elations for every $a\in\mathcal{O}$. 

Now we proceed with the same argument of the previous theorem to show
that all the elations with axis $\left[0\right]$ and center $\left(\infty\right)$
are of the previous form. Let $M$ be the vertical line $\left[c\right]$
with $c\neq0$ and let us consider two points $q_{1},q_{2}\in M\smallsetminus\left(\infty\right)$.
Let us suppose $q_{1}=\left(c,y\right)$ and $q_{2}=\left(c,y'\right)$
then the shear $\sigma_{a}$ with $a=n\left(c\right)^{-1}c*\left(y'-y\right),$
sends $q_{1}$ in $q_{2}$. Thus the group of shears is transitive
over $M\smallsetminus\left(\infty\right)$ and thus coincides with
the group of all elations with axis $\left[0\right]$ and center $\left(\infty\right)$,
i.e. i.e. $\Gamma_{\left[\left(\infty\right),\left[0\right]\right]}$.
\end{proof}
Translations and shears occur also in the octonionic realisation of
the 16-dimensional Moufang plane. We now point out a transformation
that is a collineation when formulated in the octonionic realisation,
but is not a collineation on the Okubic projective plane.
\begin{prop}
\label{thm:The-reflection-not Collineation}The reflection of the
coordinates over the Okubic plane given by $\begin{array}{c}
\left(x,y\right)\longrightarrow\left(y,x\right),\end{array}$\textup{\emph{is}} not collineation.
\end{prop}

\begin{proof}
Let us consider the image of a line $\left[s,t\right]$ through the
map that sends $\begin{array}{c}
\left(x,y\right)\longrightarrow\left(y,x\right).\end{array}$ Let us suppose that 
\begin{equation}
\left[s,t\right]=\left\{ \left(x,s*x+t\right):x\in\mathbb{\mathcal{O}}\right\} \longrightarrow\text{ \ensuremath{\left[s',t'\right]}=\ensuremath{\left\{  \left(s*x+t,x\right):x\in\mathbb{\mathcal{O}}\right\} } },
\end{equation}
 and let us determine $s'$ and $t'$. Since by definition $\left[s',t'\right]=\left\{ \left(x',s'*x'+t'\right):x'\in\mathbb{\mathcal{O}}\right\} $we
then have that 
\begin{align}
\begin{cases}
x'=s*x+t,\\
x=s'*x'+t',
\end{cases}
\end{align}
 which means
\begin{equation}
x'=s*\left(s'*x'+t'\right)+t,
\end{equation}
and thus for the (\ref{eq:x*y*x=00003Dn(x)y}) after multiplying on
the LHS for $s$, we obtain
\begin{equation}
\left(x'-t\right)*s=n\left(s\right)\left(s'*x'+t'\right),
\end{equation}
and, finally,
\begin{equation}
n\left(s\right)t'+t*s=x'*s-n\left(s\right)s'*x',
\end{equation}
which yield to a slope $s'$ that varies with $x'$ and thus it is
not a line since the slope is not fixed for all $x'$. 
\end{proof}
\begin{rem}
In case of octonions $\mathbb{O}$ reflections over the affine and
projective plane are collineations. In fact, the previous map can
be defined over the octonionic projective plane as the collineation
\begin{equation}
\rho:\begin{cases}
\left(x,y\right)\longrightarrow\left(y,x\right),\\
\left(s\right)\longrightarrow\left(s^{-1}\right),\\
\left(\infty\right)\longrightarrow\left(0\right),\\
\left(0\right)\longrightarrow\left(\infty\right),
\end{cases}
\end{equation}
 with $x,y,s\in\mathbb{O}$ which is an axial collineation of axis
$\left[1,0\right]$ and center $\left(-1\right)$ that sends 
\begin{equation}
\rho:\begin{cases}
\left[s,t\right]\longrightarrow\left[s^{-1},-s^{-1}t\right], & s\neq0\\
\left[0,t\right]\longrightarrow\left[t\right],\\
\left[t\right]\longrightarrow\left[0,t\right],\\
\left[0\right]\longrightarrow\left[\infty\right],\\
\left[\infty\right]\longrightarrow\left[0\right].
\end{cases}
\end{equation}
From an heuristic point of view the previous theorem is clear since,
for this reflection to be a collineation, it would require the existence
of an inverse at the infinity line, i.e. $\left(s\right)\longrightarrow\left(s^{-1}\right)$.
Also, note that once we try to define such collineation reading it
from the octonions from Tab. \ref{tab:Oku-Para-Octo}, i.e. defining
implicitly $s^{-1}=x$ such that read in the octonionic algebra we
would have $x\cdot s=1$, we then have two choices for the implicit
definition of $x$, i.e.
\begin{align}
\left(x*e\right)*\left(e*s\right) & =e,\text{ or }\left(s*e\right)*\left(e*x\right)=e,
\end{align}
that yield to different, even though $\tau$-conjugated, definitions
of $x$ which thus would violate the uniqueness of the extension to
the projective plane of an affine collineation. Another heuristic
reason for the lack of such collineation is that the axis of such
reflection, if it would exists as in the octonionic case $\overline{\mathscr{A}_{2}}\left(\mathbb{O}\right)$
would be the line $\left[1,0\right]$ containing all elements of the
form $\left(x,x\right)$ with $x\in\mathbb{O}$. In our case, it is
easy to verify that points $\left(x,x\right)$ are not all collinear,
e.g. the line joining the point $\left(0,0\right)$ with $\left(x,x\right)$
is given by $\left[n\left(x\right)^{-1}x*x,0\right]$ for every $x\in\mathbb{\mathcal{O}}$. 
\end{rem}

\medskip{}

\begin{rem}
The previous proposition does not mean that the set of collineations
is not transitive over $\mathcal{O}P^{2}$ since for every pair of
points $p_{1}=\left(x,y\right)$ and $p_{2}=\left(x',y'\right)$ we
can find a collineation that sends $\left(x,y\right)^{\varphi}=\left(x',y'\right)$
such as the translation $\tau_{a,b}$ with $a=x'-x$ and $b=y'-y$.
Even more the Okubic projective plane, as a Corollary of Theorem \ref{thm:Isomorphism},
is transitive on quadrangles.
\end{rem}

\subsection{\label{sec:Triality-collineations}Triality collineations }

Throught the use of the Okubic-Veronese coordinates a special set
of collineations can be easily spotted, i.e. the \emph{triality collineation}
\cite{Compact Projective} given by a cyclic permutation of the coordinates
\begin{equation}
\widetilde{t}:\left(x_{1},x_{2},x_{3};\lambda_{1},\lambda_{2},\lambda_{3}\right)\longrightarrow\left(x_{2},x_{3},x_{1};\lambda_{2},\lambda_{3},\lambda_{1}\right).
\end{equation}

\begin{prop}
The triality collineation can be read on the affine plane in the following
way:

\begin{equation}
\widetilde{t}:\begin{cases}
\left(x,y\right) & \longrightarrow\frac{1}{n\left(y\right)}\left(y,x*y\right),\,\,\,y\neq0\\
\left(x\right) & \longrightarrow\frac{1}{n\left(x\right)}\left(0,x\right),x\neq0\\
\left(x,0\right) & \longrightarrow\left(x\right),\,\,\,\\
\left(0\right) & \longrightarrow\left(\infty\right),\\
\left(\infty\right) & \longrightarrow\left(0,0\right).
\end{cases}\label{eq:triality}
\end{equation}
 In particular it induces a collineation $t\colon\mathscr{A}_{2}\left(\mathcal{O}\right)\rightarrow\mathscr{A}_{2}\left(\mathcal{O}\right)$on
the affine plane.
\end{prop}

\begin{proof}
If $y\neq0$, the image of $t\left(x,y\right)$ by the bijection (\ref{eq:correspondence})
in the projective plane is given by 
\begin{equation}
\frac{1}{n\left(y\right)}\left(y,x*y\right)\longrightarrow\frac{1}{n\left(y\right)}\left(y,x*y,\frac{y*x*y}{n\left(y\right)};\frac{n\left(x*y\right)}{n\left(y\right)},1,n\left(y\right)\right),
\end{equation}
and since $y*x*y=n\left(y\right)x$ and $n\left(x*y\right)=n\left(x\right)*n\left(y\right)$,
then the image of $t\left(x,y\right)$ is in $\mathbb{R}\left(y,x*y,x;n\left(x\right),1,n\left(y\right)\right)$
which is the image of the triality collineation $\widetilde{t}$ of
the projective point $\mathbb{R}\left(x,y,x*y;n\left(y\right),n\left(x\right),1\right)$.
With the same procedure we find the other correspondences. 
\end{proof}
\begin{figure}
\centering{}\includegraphics[scale=0.18]{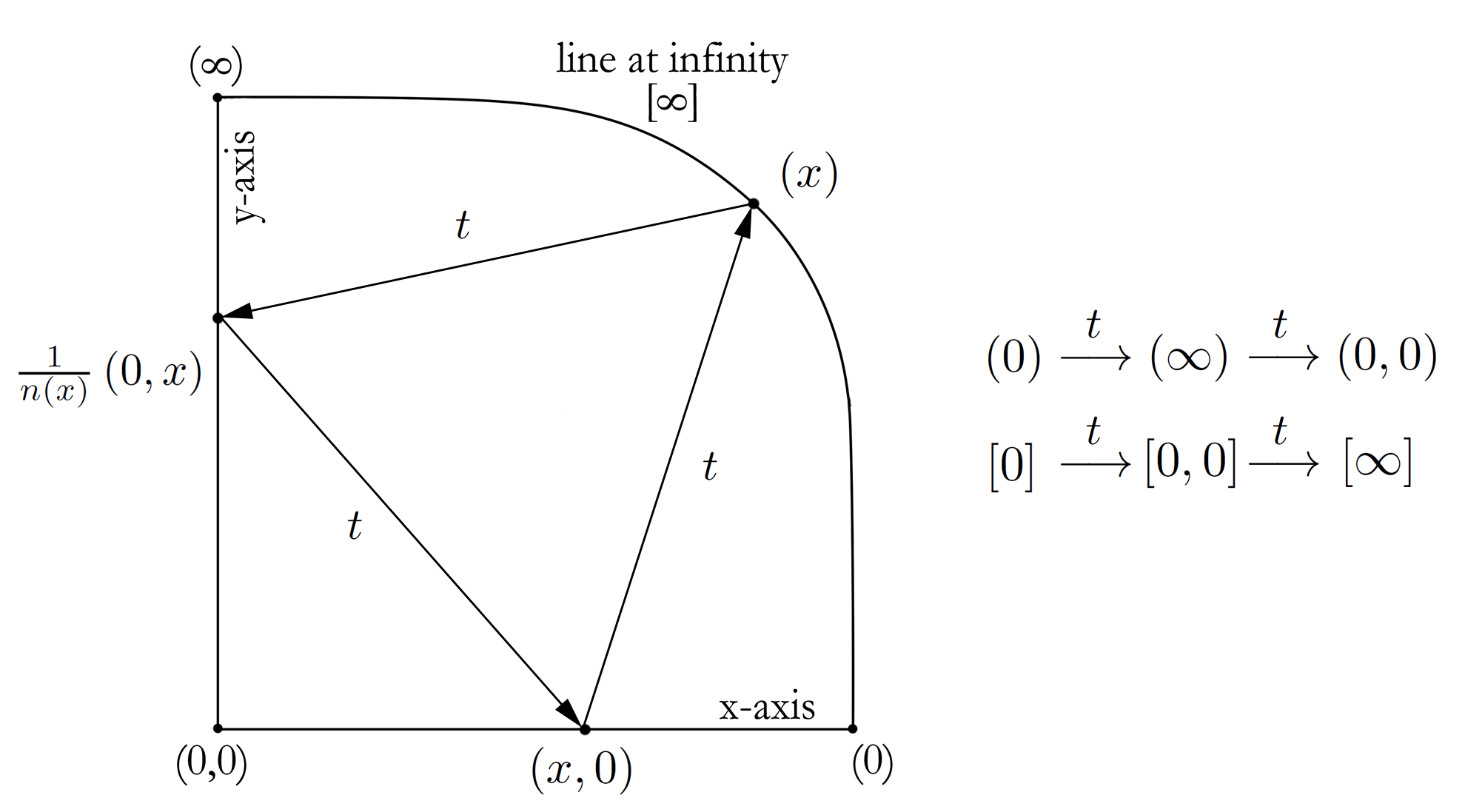}\caption{\label{fig:Action-on-the}Action on the affine plane $\mathscr{A}_{2}\left(\mathcal{O}\right)$
of the triality collineation defined in (\ref{eq:triality}).}
\end{figure}

\begin{rem}
As shown in Fig (\ref{fig:Action-on-the}) the triality collineation
$t$ sends the line at infinity $\left[\infty\right]$ into the line
$\left[0\right]$, while the $y$ axis $\left[0\right]$ is sent into
the $x$ axis $\left[0,0\right]$; finally the $x$ axis $\left[0,0\right]$
is sent into the line at infinity $\left[\infty\right]$. This phenomenon
is the dual, of what happens, in the reverse order, for the three
points $\left(0,0\right)$,$\left(0\right)$ and $\left(\infty\right)$. 
\end{rem}

\section{\label{sec:Three-realizations-of}Three realizations of the 16-dimensional
Moufang plane }

In section \ref{sec:Affine-and-projective}, we have constructed three
projective planes coming from three division algebras, by modifications
of Veronese-type formulas. In the preceding section \ref{sec:Collineations},
we explicitly constructed the primary families of collineations for
the Okubonic plane and highlighted certain distinctive features of
the plane. More specifically, we observed that the points $(0,0)$,
$(x,x)$, and $(y,y)$ are not collinear -as it happens in projective
planes obtained over Hurwitz algebras-, and the transformation from
$(x,y)\longrightarrow(y,x)$ does not constitute a collineation. Despite
these distinctions, in Theorem \ref{thm:Isomorphism}, we construct
two collineations, i.e. (\ref{eq:isomorfPOPO}) and (\ref{eq:isomorfPOPO-1})
that prove the three planes to be projectively isomorphic. Furthermore,
we will show that such collineations are isometries. As a consequence,
there exists a complete equivalence among the Okubonic, octonionic,
and paraoctonic planes that we will extensively discuss in section
\ref{sec:Discussions-and-verifications}.

\subsection{Isomorphism between Okubic and the Cayley plane }

In the context of projective spaces, an isomorphism refers to a bijection
between the points of the spaces that preserves the incidence relations.
We have the following 
\begin{thm}
\label{thm:Isomorphism}The Okubic projective $\mathcal{O}P^{2}$
plane is isomorphic to the octonionic projective plane $\mathbb{O}P^{2}$.
\end{thm}

\begin{proof}
Consider the following bijective map (see sec. \ref{subsec:Conjugation-and-the})
defined over the real Okubo algebra given by

\begin{align}
x\longrightarrow\overline{x} & =\left\langle x,e\right\rangle e-x,\\
x\longrightarrow\tau\left(x\right) & =\left\langle x,e\right\rangle e-x*e,
\end{align}
where $*$ is the Okubic product. Notice that $\tau$, as bijective
maps over the octonions, is an order three automorphism that realizes
the Okubo algebra as a Petersson algebra since 
\begin{align}
x*y & =\tau\left(\overline{x}\right)\cdot\tau^{2}\left(\overline{y}\right),
\end{align}
for every $x,y\in\mathbb{O}$. Let the Okubic projective plane be
$\mathcal{O}P^{2}=\left\{ \mathscr{P}_{\mathbb{\mathcal{O}}},\mathscr{L}_{\mathcal{O}},\mathscr{R}_{\mathcal{O}}\right\} $
and consider the bijective map $\Phi:\mathcal{O}P^{2}\longrightarrow\mathbb{O}P^{2}$
given by
\begin{equation}
\Phi:\begin{cases}
\left(x,y\right) & \longrightarrow\left(\tau^{2}\left(\overline{x}\right),y\right),\\
\left(s\right) & \longrightarrow\left(\tau\left(\overline{s}\right)\right),\\
\left(\infty\right) & \longrightarrow\left(\infty\right),\\
\left[s,t\right] & \longrightarrow\left[\tau\left(\overline{s}\right),t\right],\\
\left[c\right] & \longrightarrow\left[\tau^{2}\left(\overline{c}\right)\right],\\
\left[\infty\right] & \longrightarrow\left[\infty\right].
\end{cases}\label{eq:isomorfPOPO}
\end{equation}
If we call the image incidence plane $\Phi\left(\mathcal{O}P^{2}\right)=\left\{ \mathscr{R}_{\mathcal{O}}^{\Phi},\mathscr{L}_{\mathcal{O}}^{\Phi},\mathscr{R}_{\mathcal{O}}^{\Phi}\right\} $,
then we notice that the octonionic projective plane is given by $\mathbb{O}P^{2}=\left\{ \mathscr{P}_{\mathbb{\mathcal{O}}}^{\Phi},\mathscr{L}_{\mathcal{O}}^{\Phi},\mathscr{R}_{\mathbb{O}}\right\} .$
To show the projective isomorphism and complete the theorem we need
to show that $\mathscr{R}_{\mathbb{O}}\cong\mathscr{R}_{\mathcal{O}}^{\Phi}$,
in other words that every point in the Okubic plane $\left(x,y\right)$
is incident to an Okubic line $\ell$ if and only if the image point
$\left(x,y\right)^{\Phi}$ is incident to the image of the octonion
line $\ell^{\Phi}$, i.e. $\Phi\left(\left(x,y\right)\right)\in\Phi\left(\ell\right)$.
By definition of the Okubic projective plane 

\begin{align}
\left(x,y\right) & \in\left[s,t\right]=\left\{ y=s*x+t\right\} ,\\
\left(s\right) & \in\left[s,t\right],\\
\left(\infty\right) & \in\left[\infty\right],
\end{align}
for every $x,y,s\in\mathcal{O}$. But, since the image of the line
$\left[s,t\right]$ is 
\begin{equation}
\left[\tau\left(\overline{s}\right),t\right]=\left\{ \left(x,y\right)\in\mathbb{O}:y=\tau\left(\overline{s}\right)\cdot\tau^{2}\left(\overline{x}\right)+t\right\} ,
\end{equation}
 and since 
\begin{equation}
s*x=\tau\left(\overline{s}\right)\cdot\tau^{2}\left(\overline{x}\right),
\end{equation}
we then have that 
\begin{align}
\left(x,y\right)^{\Phi} & =\left(\tau^{2}\left(\overline{x}\right),y\right)\in\left[\tau\left(\overline{s}\right),t\right]=\left[s,t\right]^{\Phi},\\
\left(s\right)^{\Phi} & =\left(\tau\left(\overline{s}\right)\right)\in\left[\tau\left(\overline{s}\right),t\right]=\left[s,t\right]^{\Phi},\\
\left(\infty\right) & \in\left[\infty\right],
\end{align}
which thus concludes the proof of the theorem.
\end{proof}
In the previous theorem we explicitly found an isomorphism between
the completion of the affine plane over the Okubo algebra and that
over the octonions. For practical reason it is also useful to have
the isomorphism $\widetilde{\Phi}:\mathcal{O}P^{2}\longrightarrow\mathbb{O}P^{2}$
developed for the Veronese formalism, i.e. between Veronese Okubic
and octonionic vectors. The isomorphism between the Okubic Veronese
vectors and octonionic Veronese vectors is given by
\begin{equation}
\begin{cases}
\left(x,y,x*y;n\left(y\right),n\left(x\right),1\right)\longrightarrow\left(\tau^{2}\left(\overline{x}\right),y,y\cdot\overline{\tau^{2}\left(\overline{x}\right)};n\left(y\right),n\left(x\right),1\right),\\
\left(0,0,x;n\left(x\right),1,0\right)\longrightarrow\left(0,0,\tau^{2}\left(\overline{x}\right);n\left(x\right),1,0\right),\\
\left(0,0,0;1,0,0\right)\longrightarrow\left(0,0,0;1,0,0\right),
\end{cases}
\end{equation}
where the first vectors are Veronese under conditions (\ref{eq:Okubo Ver-1})
and (\ref{eq:Okubo Ver-2}), while the image vectors are Veronese
under conditions (\ref{eq:Ver-1-Oct}) and (\ref{eq:Ver-2-Oct}) that
involves conjugation and octonionic product.

\subsection{Isomorphism with the para-octonionic plane}

Recall that a point of the paraoctonionic affine plane $\mathscr{A}_{2}\left(p\mathbb{O}\right)$
plane is given by a pair of elements $\left(x,y\right)$ with $x$,$y\in\left\{ p\mathbb{O}\right\} $,
while a \emph{line} of slope $s\in p\mathbb{O}$ and offset $t\in p\mathbb{O}$
is the set $\left[s,t\right]=\left\{ \left(x,s\bullet x+t\right):x\in p\mathbb{O}\right\} $
and, of course, we say that a point $\left(x,y\right)\in\mathscr{A}_{2}\left(p\mathbb{O}\right)$
is \emph{incident }to a line $\left[s,t\right]\subset\mathscr{A}_{2}\left(p\mathbb{O}\right)$
if belongs to such line, i.e. $\left(x,y\right)\in\left[s,t\right]$.
If the previous definitions define a para-octonionic affine plane,
a projective plane can be directly defined through the Veronese conditions
(\ref{eq:Ver-1-paraOct-1}) and (\ref{eq:Ver-2-paraOct-1}), i.e.,

\begin{align}
\lambda_{1}x_{1} & =x_{2}\bullet x_{3},\,\,\lambda_{2}x_{2}=x_{3}\bullet x_{1},\,\,\lambda_{3}x_{3}=x_{1}\bullet x_{2}\label{eq:Ver-1-paraOct}\\
n\left(x_{1}\right) & =\lambda_{2}\lambda_{3},\,n\left(x_{2}\right)=\lambda_{3}\lambda_{1},n\left(x_{3}\right)=\lambda_{1}\lambda_{2}.\label{eq:Ver-2-paraOct}
\end{align}

An explicit isomorphism between the Okubic projective plane $\mathcal{O}P^{2}$
and the para-octonionic projective plane $p\mathbb{O}P^{2}$  is obtained
considering the bijective map $p\Phi:\mathcal{O}P^{2}\longrightarrow p\mathbb{O}P^{2}$
given by
\begin{equation}
p\Phi:\begin{cases}
\left(x,y\right) & \longrightarrow\left(\tau^{2}\left(x\right),y\right),\\
\left(s\right) & \longrightarrow\left(\tau\left(s\right)\right),\\
\left(\infty\right) & \longrightarrow\left(\infty\right),\\
\left[s,t\right] & \longrightarrow\left[\tau\left(s\right),t\right],\\
\left[c\right] & \longrightarrow\left[\tau^{2}\left(c\right)\right],\\
\left[\infty\right] & \longrightarrow\left[\infty\right].
\end{cases}\label{eq:isomorfPOPO-1}
\end{equation}
The proof that the map given in (\ref{eq:isomorfPOPO-1}) is a collineation
adheres closely to the steps outlined in Theorem \ref{thm:Isomorphism}.
This is expected, given the similarity between the maps. The sole
distinction between para-octonions $p\mathbb{O}$ and octonions $\mathbb{O}$is
the presence of a paraunit in the former, as opposed to a unit in
the latter, and the fact that while the former is merely flexible,
the latter is alternative.

\subsection{Isometries}

Theorem \ref{thm:Isomorphism} and its para-octonionic counterpart
ensures projective isomorphism between the three planes $\mathcal{O}P^{2}$,
$\mathbb{O}P^{2}$ and $p\mathbb{O}P^{2}$. If projective spaces are
isomorphic, they share the same incidence relations between points
and lines. However, this does not imply that the distances or angles
between points are preserved and this is of high importance in our
case since, according to Lie theory, it is well-known that the collineation
group of the octonionic plane is the minimally non-compact real form
of the exceptional Lie group $\text{E}_{6}$, namely $\text{E}_{6(-26)}$,
while the group of elliptic motions, i.e. isometries, over the octonionic
projective plane is its maximal compact subgroup, namely $\text{F}_{4(-52)}$
(see \cite{corr Notes Octo}). By the map in (\ref{eq:isomorfPOPO})
we have an Okubic realization of both the Lie groups $\text{F}_{4(-52)}$
and $\text{E}_{6(-26)}$. Thus, we can write the homogeneous space
presentation of the compact Cayley-Moufang plane as $\text{F}_{4(-52)}/\text{Spin}\left(9\right)$,a
16-dimensional symmetric coset. In fact, we have the following
\begin{thm}
\label{thm:The-map-Isometric}The map $\Phi$ defined in (\ref{eq:isomorfPOPO})
is an isometry 
\end{thm}

\begin{proof}
By construction the Okubic plane comes equipped with the following
distance:
\begin{equation}
d_{\mathcal{O}}\left(p_{1},p_{2}\right)=n\left(x_{1}-x_{2}\right)^{2}+n\left(y_{1}-y_{2}\right)^{2},
\end{equation}
with $p_{1}=\left(x_{1},y_{1}\right)\in\mathscr{A}^{2}\left(\mathcal{O}\right)$
and $p_{2}=\left(x_{2},y_{2}\right)\in\mathscr{A}^{2}\left(\mathcal{O}\right)$.
By its very construction the octonionic plane comes with the following
distance:
\begin{equation}
d_{\mathbb{O}}=n\left(x_{1}-x_{2}\right)^{2}+n\left(y_{1}-y_{2}\right)^{2},
\end{equation}
with $p_{1}=\left(x_{1},y_{1}\right)\in\mathscr{A}^{2}\left(\mathbb{O}\right)$
and $p_{2}=\left(x_{2},y_{2}\right)\in\mathscr{A}^{2}\left(\mathbb{O}\right)$.Then,
the images by $\Phi$ of the two points $p_{1}$ and $p_{2}$ are
given by
\begin{equation}
p_{1}^{\Phi}=\left(\tau^{2}\left(\overline{x}_{1}\right),y_{1}\right),p_{2}^{\Phi}=\left(\tau^{2}\left(\overline{x}_{2}\right),y_{2}\right).
\end{equation}
 Therefore the octonionic distance between $p_{1}^{\Phi}$ and $p_{2}^{\Phi}$
is given by 
\begin{equation}
d_{\mathbb{O}}\left(p_{1}^{\Phi},p_{2}^{\Phi}\right)=n\left(\tau^{2}\left(\overline{x}_{1}\right)-\tau^{2}\left(\overline{x}_{2}\right)\right)^{2}+n\left(y_{1}-y_{2}\right)^{2},
\end{equation}
 but since all automorphism of Hurwitz and para-Hurwitz algebras as
isometries and $\tau^{2}$ is an automorphism we have
\begin{align}
d_{\mathbb{O}}\left(p_{1}^{\Phi},p_{2}^{\Phi}\right) & =n\left(\tau^{2}\left(\overline{x}_{1}-\overline{x}_{2}\right)\right)^{2}+n\left(y_{1}-y_{2}\right)^{2}\\
 & =n\left(\overline{x}_{1}-\overline{x}_{2}\right)^{2}+n\left(y_{1}-y_{2}\right)^{2}\\
 & =n\left(x_{1}-x_{2}\right)^{2}+n\left(y_{1}-y_{2}\right)^{2}\\
 & =d_{\mathcal{O}}\left(p_{1},p_{2}\right),
\end{align}
thus completing the proof.
\end{proof}
Similar and straightforward proof can be given for the map $p\Phi$
for the para-octonionic case.

\subsection{Collineation groups}

As a corollary of Theorem \ref{thm:Isomorphism} we have that the
Lie group of collineations of the Okubic projective plane is the following 
\begin{cor}
\label{cor:Collineations}The group of collineations of the Okubic
projective plane is $\text{E}_{6\left(-26\right)}$. 
\end{cor}

\begin{proof}
Let $\Phi:\mathcal{O}P^{2}\longrightarrow\mathbb{O}P^{2}$ be the
isomorphism in (\ref{eq:isomorfPOPO}) and $\Phi^{-1}:\mathbb{O}P^{2}\longrightarrow\mathcal{O}P^{2}$
its inverse given by 
\begin{align}
\Phi^{-1} & :\begin{cases}
\left(x,y\right) & \longrightarrow\left(\tau\left(\overline{x}\right),y\right)\\
\left(s\right) & \longrightarrow\left(\tau^{2}\left(\overline{s}\right)\right)\\
\left(\infty\right) & \longrightarrow\left(\infty\right)\\
\left[s,t\right] & \longrightarrow\left[\tau^{2}\left(\overline{s}\right),t\right]\\
\left[c\right] & \longrightarrow\left[\tau\left(\overline{c}\right)\right]\\
\left[\infty\right] & \longrightarrow\left[\infty\right]
\end{cases},
\end{align}
where $x,y,s,t,c\in\mathbb{O}$. Then, since both $\Phi$ and $\Phi^{-1}$
send lines to lines, for every collineation of the octonionic projective
plane $\gamma\in\text{Aut}\left(\mathbb{O}P^{2}\right)$, then the
composition of collineations
\begin{equation}
\widetilde{\gamma}=\Phi^{-1}\gamma\Phi,\label{eq:isomorf-collinea}
\end{equation}
 is a collineation of the Okubic projective plane. Conversely any
collineation $\widetilde{\gamma}\in\text{Aut}\left(\mathcal{O}P^{2}\right)$
induces a collineation 
\begin{equation}
\gamma=\Phi\widetilde{\gamma}\Phi^{-1},
\end{equation}
in $\text{Aut}\left(\mathbb{O}P^{2}\right)$. Moreover, it is clear
that
\begin{equation}
\widetilde{\gamma}\circ\widetilde{\delta}=\left(\Phi^{-1}\gamma\Phi\right)\left(\Phi^{-1}\delta\Phi\right)=\Phi^{-1}\left(\gamma\delta\right)\Phi,
\end{equation}
so that the two collineation groups are identical $\text{Aut}\left(\mathbb{O}P^{2}\right)\cong\text{Aut}\left(\mathcal{O}P^{2}\right)$,
i.e. $\text{Aut}\left(\mathcal{O}P^{2}\right)\cong\text{E}_{6\left(-26\right)}$
.
\end{proof}
For the sake of completeness, we now recover the collineation $\widetilde{\gamma}=\Phi^{-1}\gamma\Phi$
corresponding to the octonionic collineation $\gamma$ given by the
reflection $\left(x,y\right)\longrightarrow\left(y,x\right)$. By
(\ref{eq:isomorf-collinea}), we have that the octonionic reflection
given by switching coordinates is on the Okubic plane given by the
collineation
\begin{equation}
\widetilde{\gamma}:\left(x,y\right)\longrightarrow\left(\tau\left(\overline{y}\right),\tau^{2}\left(\overline{x}\right)\right).
\end{equation}

Moreover, given that the group of elliptic motion of the octonionic
plane is $\text{F}_{4\left(-52\right)}$, we have the following corollary
of Theorem \ref{thm:The-map-Isometric} 
\begin{cor}
\label{cor:Isometry}The group of elliptic motion of the Okubic projective
plane is $\text{F}_{4\left(-52\right)}$. 
\end{cor}

Corollary \ref{cor:Collineations} and \ref{cor:Isometry} state the
existence of an Okubonic geometric realization of exceptional Lie
groups $\text{E}_{6\left(-26\right)}$ and $\text{F}_{4\left(-52\right)}$
which, to our knowledge was never pointed out.

\subsection{\label{subsec:Stabilizer-of-a}$G_{2}$ as stabilizer of a quadrangle}

Another exceptional Lie group with direct geometrical significance
in the octonionic plane is $G_{2\left(-14\right)}$. This exceptional
Lie group is recognized as the group of automorphisms of the octonions,
denoted as $\text{Aut}\left(\mathbb{O}\right)=G_{2\left(-14\right)}$.
However, this is also the group $\Gamma\left(\diamondsuit,\mathbb{O}\right)$
of collineations that fix every point of a quadrangle \cite{corr Notes Octo,Compact Projective}
of the projective octonionic plane. Given Theorems \ref{thm:Isomorphism}
and \ref{thm:The-map-Isometric} our objective is to identify an Okubic
realization of $G_{2\left(-14\right)}$. To this end, we examine the
subgroups of collineations $\Gamma\left(\triangle,\mathbb{\mathcal{O}}\right)$
and $\Gamma\left(\diamondsuit,\mathbb{\mathcal{O}}\right)$. Specifically,
the former represents the group of collineations that preserve every
point of the triangle $\triangle=\left\{ \left(0,0\right),\left(0\right),\left(\infty\right)\right\} $,
while the latter is the group that maintains every point of the quadrangle
$\diamondsuit=\left\{ \left(0,0\right),\left(e,e\right),\left(0\right),\left(\infty\right)\right\} $. 
\begin{prop}
\label{prop:Stabilizer Tri}The group $\Gamma\left(\triangle,\mathbb{\mathcal{O}}\right)$
of collineations that fix every point of $\triangle$ are transformations
of this form 
\begin{align}
\left(x,y\right) & \mapsto\left(A\left(x\right),B\left(y\right)\right)\\
\left(s\right) & \mapsto\left(C\left(s\right)\right)\\
\left(\infty\right) & \mapsto\left(\infty\right)
\end{align}
where $A,B$ and $C$ are automorphism for the sum over $\mathbb{\mathcal{O}}$
and in respect to multiplication they satisfy 
\begin{equation}
B\left(s*x\right)=C\left(s\right)*A\left(x\right).\label{eq:B(xs)=00003DC(s)A(x)}
\end{equation}
\end{prop}

\begin{proof}
A collineation that fixes $\left(0,0\right)$, $\left(0\right)$ and
$\left(\infty\right)$, thus it also leaves invariant the $x$-axis
and $y$-axis. Moreover, since the incidence relations must be preserved,
maps all lines parallel to the $x$-axis and the $y$-axis to lines
parallel to the $x$-axis and the $y$-axis. Then, the first coordinate
is the image of a function that does not depend on $y$ and the second
coordinate is image of a fuction that does not depend by $x$, i.e.
$\left(x,y\right)\mapsto\left(A\left(x\right),B\left(y\right)\right)$
and $\left(s\right)\mapsto\left(C\left(s\right)\right)$. Now consider
the image of a point on the line $\left[s,t\right]$. The point is
of the form $\left(x,s*x+t\right)$ and its image goes to 
\begin{equation}
\left(x,s*x+t\right)\mapsto\left(A\left(x\right),B\left(s*x+t\right)\right).
\end{equation}
In order this to be a collineation, the points of $\left[s,t\right]$
must all belong to a line that, setting $x=0$, passes through the
points $p_{1}=\left(0,B\left(t\right)\right)$ and $p_{2}=\left(C\left(s\right)\right)$,
e.g. $\left[C\left(s\right),B\left(t\right)\right]$. Every line $\left(A\left(x\right),B\left(s*x+t\right)\right)$
passing through $p_{1}$ and $p_{2}$ must satisfy the condition 
\begin{equation}
B\left(s*x+t\right)=C\left(s\right)*A\left(x\right)+B\left(t\right).\label{eq:Condition B=00003D00003DCA}
\end{equation}
Given (\ref{eq:Condition B=00003D00003DCA}), if $B$ is an automorphism
with respect to the sum over $\mathbb{\mathcal{O}}$, then $B\left(s*x\right)=C\left(s\right)*A\left(x\right)$.
Conversely if $B\left(s*x\right)=C\left(s\right)*A\left(x\right)$
is true than $B\left(s*x+t\right)=B\left(s*x\right)+B\left(t\right)$
and $B$ is an automorphism with respect to the sum. 
\end{proof}
A further corollary of Theorem \ref{thm:Isomorphism} is that the
Okubic projective plane, being the 16-dimensional Moufang planes,
has a collineation group transitive on quadrangles. Without loss of
generality we can thus consider the quadrangle $\diamondsuit$ given
by the points $\left(0,0\right)$, $\left(e,e\right)$, $\left(0\right)$
and $\left(\infty\right)$, that is $\diamondsuit=\triangle\cup\left\{ \left(e,e\right)\right\} $,
and consider the collineations that fix the set $\diamondsuit$.

For this purpose, it is important to note that a relation analogous
to (\ref{eq:B(xs)=00003DC(s)A(x)}) holds in the case of both para-octonions
and octonions. Specifically, in the octonionic case, we have
\begin{equation}
B\left(s\cdot x\right)=C\left(s\right)\cdot A\left(x\right),
\end{equation}
for all $x,s\in\mathbb{O}$. This is particularly relevant since,
in the case of octonions, the relation simplifies if we require the
collineations to also stabilize the point $\left(e,e\right)$, i.e.
to stabilize the non-degenerate quadrangle $\diamondsuit=\triangle\cup\left\{ \left(e,e\right)\right\} .$
In this scenario, the previous formula transforms into
\begin{equation}
A\left(s\cdot x\right)=A\left(s\right)\cdot A\left(x\right),
\end{equation}
which implies that the stabilizer of the non-degenerate quadrangle
is isomorphic to the automorphism group of the octonions, denoted
as $G_{2\left(-14\right)}.$ Switching back to the Okubic algebra
from the previous result, we then arrive at the following statement:
\begin{thm}
The exceptional Lie group $G_{2\left(-14\right)}$ has the following
realisation 

\begin{equation}
\begin{array}{cc}
G_{2\left(-14\right)}\cong\left\{ \left(A,B,C\right)\in\text{Spin}\left(8\right):\right. & B\left(s*x\right)=C\left(s\right)*A\left(x\right),\\
 & \left.A\left(e\right)=B\left(e\right)=C\left(e\right)=e\right\} ,
\end{array}\label{eq:G2-Ok}
\end{equation}
or, equivalently

\begin{equation}
\begin{array}{cc}
G_{2\left(-14\right)}\cong\left\{ \left(A,B,C\right)\in\text{Tri}\left(\mathcal{O}\right):\right. & B\left(e\ast x\right)=e\ast A\left(x\right),\\
 & \left.B\left(x\ast e\right)=C\left(x\right)\ast e\right\} ,
\end{array}\label{eq:G2-Ok-1}
\end{equation}
where $x,s\in\mathbb{\mathcal{O}}$ and $e*e=e\in\mathbb{\mathcal{O}}$.
\end{thm}

We give here two independent proofs of the same statement. The first
is a Corollary of the isomorphism in Theorem \ref{thm:Isomorphism},
relying on the fact that the stabilizer of a non-degenerate quadrangle
on the octonionic projective plane is $\text{G}_{2\left(-14\right)}.$
The second is a proof on Lie theory, that does not rely on the knowledge
of the geometry of the octonionic plane. 
\begin{proof}
From the previous proposition the stabilizer of the triangle $\triangle$
is given by the group of triples $\left(A,B,C\right)\in\text{Spin}\left(8\right)$
such that
\begin{equation}
B\left(s*x\right)=C\left(s\right)*A\left(x\right),
\end{equation}
with $x,s\in\mathbb{\mathcal{O}}$. To stabilize $\diamondsuit=\triangle\cup\left\{ \left(e,e\right)\right\} $,
we have to impose 
\begin{equation}
\left(e,e\right)\mapsto\left(A\left(e\right),B\left(e\right)\right)=\left(e,e\right),
\end{equation}
and since $e*e=e$, then $C\left(e\right)=e$ and $A\left(e\right)=B\left(e\right)=C\left(e\right)=e$,
obtaining the RHS of (\ref{eq:G2-Ok}). Knowing that the stabilizer
of a non-degenerate quadrangle of the compact 16-dimensional Moufang
plane is $\text{G}_{2\left(-14\right)}$, we then obtain the identification
with the LHS of (\ref{eq:G2-Ok}).
\end{proof}
We now proceed with the second proof of Theorem 29.
\begin{proof}
Recall that the triality group Tri$\left(\mathcal{O}\right)$ of the
real Okubo algebra $\mathcal{O}$ is defined as
\begin{equation}
\text{Tri}\left(\mathcal{O}\right):=\left\{ A,B,C\in\text{Spin}\left(\mathcal{O}\right):B\left(s\ast x\right)=C\left(s\right)\ast A\left(x\right),~\forall s,x\in\mathcal{O}\right\} ,\label{def-Tri}
\end{equation}
and that, as proved in \cite{KMRT}, $\text{Tri}\left(\mathcal{O}\right)\simeq\text{Spin}\left(8\right)\simeq\text{Spin}\left(\mathcal{O}\right).$
Let us consider the action of triality (\ref{def-Tri}) in three cases:
the first where $s=x=e$, for which 
\begin{equation}
B\left(e\right)=C\left(e\right)\ast A\left(e\right),\label{eq:1}
\end{equation}
 the second with $s\in\mathcal{O}$ and fixed $x=e$, i.e., 
\begin{equation}
B\left(s*e\right)=C\left(s\right)\ast A\left(e\right),\label{eq:2}
\end{equation}
finally, the case $x\in\mathcal{O}$ and $s=e$ for which
\begin{equation}
B\left(e*x\right)=C\left(e\right)\ast A\left(x\right).\label{eq:3}
\end{equation}
Now, we want to determine the subgroup of Tri$\left(\mathcal{O}\right)$
defined by the following constraints
\begin{eqnarray}
B\left(e\ast x\right) & = & e\ast A\left(x\right),\label{con-1}\\
B\left(x\ast e\right) & = & C\left(x\right)\ast e.\label{con-2}
\end{eqnarray}
Since $\mathcal{O}$ is a division algebra, (\ref{eq:3}) and the
constraint (\ref{con-1}), as well as (\ref{eq:2}) and the constraint
(\ref{con-2}), imply
\begin{equation}
C\left(e\right)=e=A\left(e\right),\label{conss}
\end{equation}
which in turn imply, by (\ref{eq:1}), that $B\left(e\right)=e\ast e=e.$
Thus one can reformulate the constraints (\ref{con-1}) and (\ref{con-2}),
for any subgroup of Tri$\left(\mathcal{O}\right)\simeq\text{Spin}\left(8\right)$,
with (\ref{conss}). 

A well-known theorem by Dynkin (see Th. 1.5 of \cite{Dynkin}) states
that a maximal (and non-symmetric) embedding of $\text{SU}_{3}=$Aut$\left(\mathcal{O}\right)$
into $\text{Spin}\left(8\right)$ exists such that all 8-dimensional
irreducible representations (henceforth denoted with the term ``irreprs.'')
of $\text{Spin}\left(8\right)$ stay irreducible in $\text{SU}_{3}$,
all reducing to the same adjoint representation. By using the Dynkin
labels to identify the representations, it holds that 
\begin{gather}
\text{Spin}\left(8\right)\underset{\text{\text{max,ns}}}{\supset}\text{SU}_{3},\\
\begin{array}{l}
\left(1,0,0,0\right)=\left(1,1\right),\\
\left(0,0,0,1\right)=\left(1,1\right),\\
\left(0,0,1,0\right)=\left(1,1\right),
\end{array}\nonumber 
\end{gather}
where we adopted the conventions of \cite{Yamatsu}, and the subscripts
``max'', ``s'' and ``ns'' respectively stand for maximal, symmetric
and non-symmetric. Thus, for the triality of Spin$\left(8\right)$,
the adjoint irrepr. $\left(1,1\right)$ of $\text{SU}_{3}$, for which
the basis (\ref{eq:definizione i ottonioniche}) of the Okubo algebra
$\mathcal{O}$ provides a realization
\begin{equation}
\mathcal{O}\simeq\left\{ e,\text{i}_{1},..,\text{i}_{7}\right\} \simeq\left(1,1\right)\text{~of~}\text{SU}_{3},
\end{equation}
can be mapped to any of the three 8-dimensional irreprs. of $\text{Spin}\left(8\right)$
and, with no loss of generality, up to triality of $\text{Spin}\left(8\right)$
, one can identify
\begin{equation}
\begin{array}{l}
C\left(\mathcal{O}\right):=\left\{ C\left(e\right),C\left(\text{i}_{1}\right),..,C\left(\text{i}_{7}\right)\right\} \simeq\left(1,0,0,0\right)~\text{of~Spin}\left(8\right),\\
A\left(\mathcal{O}\right):=\left\{ A\left(e\right),A\left(\text{i}_{1}\right),..,A\left(\text{i}_{7}\right)\right\} \simeq\left(0,0,0,1\right)~\text{of~Spin}\left(8\right),\\
B\left(\mathcal{O}\right):=\left\{ B\left(e\right),B\left(\text{i}_{1}\right),..,B\left(\text{i}_{7}\right)\right\} \simeq\left(0,0,1,0\right)~\text{of~Spin}\left(8\right).
\end{array}
\end{equation}

We now implement the first constraint of (\ref{conss}), namely of
$C\left(e\right)=e$. The largest subgroup of $\text{Spin}\left(8\right)$
allowing $C\left(e\right)=e$ is its maximal (and symmetric) subgroup
Spin$\left(7\right)$, for which it holds that
\begin{gather}
\text{Spin}(8)\underset{\text{\text{max,s}}}{\supset}\text{Spin}\left(7\right)\ni A,B,C:\left\{ \begin{array}{l}
\left(1,0,0,0\right)=\underset{C\left(e\right)=e}{\left(0,0,0\right)}\oplus\underset{\left\{ C\left(\text{i}_{1}\right),..,C\left(\text{i}_{7}\right)\right\} }{\left(1,0,0\right)},\\
\left(0,0,0,1\right)=\underset{\left\{ A\left(e\right),A\left(\text{i}_{1}\right),..,A\left(\text{i}_{7}\right)\right\} }{\left(0,0,1\right)},\\
\left(0,0,1,0\right)=\underset{\left\{ B(e),B(\text{i}_{1}),..,B(\text{i}_{7})\right\} }{\left(0,0,1\right)},
\end{array}\right.
\end{gather}
where $\left(1,0,0\right)$ is the 7-dimensional fundamental (\textit{vector})
irrepr. of $\text{Spin}\left(7\right)$ and $\left(0,0,1\right)$
denotes its 8-dimensional (\textit{spinor}) irrepr.. Next, one must
impose the second constraint of (\ref{conss}). This can be implemented
by a further symmetry breaking implying $A\left(e\right)=e$, which,
together to $C\left(e\right)=e$, implies also that $B\left(e\right)=e$.
The largest subgroup of $\text{Spin}\left(7\right)$ allowing for
an action of this kind is $G_{2(-14)}$, which is a maximal and non-symmetric
subgroup of $\text{Spin}\left(7\right)$ itself : 
\begin{gather}
\text{Spin}\left(7\right)\underset{\text{\text{max,ns}}}{\supset}G_{2(-14)}\ni A,B,C:\\
\left\{ \begin{array}{l}
\left(0,0,0\right)=\underset{C\left(e\right)=e}{\left(0,0\right)},\\
\left(1,0,0\right)=\underset{\left\{ C\left(\text{i}_{1}\right),..,C\left(\text{i}_{7}\right)\right\} }{\left(0,1\right)},\\
\left(0,0,1\right)=\underset{A\left(e\right)=e}{\left(0,0\right)}\oplus\underset{\left\{ A\left(\text{i}_{1}\right),..,A\left(\text{i}_{7}\right)\right\} }{\left(0,1\right)},\\
\left(0,0,1\right)=\underset{B\left(e\right)=e}{\left(0,0\right)}\oplus\underset{\left\{ B\left(\text{i}_{1}\right),..,B\left(\text{i}_{7}\right)\right\} }{\left(0,1\right)}.
\end{array}\right.
\end{gather}
Thus, the (largest) subgroup of Tri$\left(\mathcal{O}\right)\simeq$$\text{Spin}\left(8\right)$
defined by the constraints (\ref{conss}) (or, equivalently, by the
constraints (\ref{con-1}) and (\ref{con-2})) is $G_{2(-14)}$, which
is next-to-maximal (and symmetric) in $\text{Spin}\left(8\right)$,
being determined by the chain of two maximal embeddings:
\begin{equation}
\text{Spin}\left(8\right)\supset_{\text{max,s}}\text{Spin}\left(7\right)\supset_{\text{max,ns}}G_{2(-14)}.
\end{equation}
\end{proof}

\section{\label{sec:Discussions-and-verifications}Discussions and verifications}

Let $A$ be an algebra, and let $x,y,z$ be elements of this algebra.
It is well known that the validity of Moufang identities (\ref{eq:MoufangIdent3})
in the algebra, i.e. 
\begin{align}
\left(\left(x\cdot y\right)\cdot x\right)\cdot z & =x\cdot\left(y\cdot\left(x\cdot z\right)\right),\\
\left(\left(z\cdot x\right)\cdot y\right)\cdot x & =z\cdot\left(x\cdot\left(y\cdot x\right)\right),\\
\left(x\cdot y\right)\cdot\left(z\cdot x\right) & =x\cdot\left(\left(y\cdot z\right)\cdot x\right),\label{eq:MoufangIdent3-1}
\end{align}
are linked with the Moufang properties of projective plane over the
algebra \cite[Sec. 12.15]{Mou35,Compact Projective}. Furthermore,
as an immediate corollary, if the algebra is unital, then, setting
$y=1$ the Moufang identities imply alternativity, resulting in 
\begin{align}
\left(x\cdot x\right)\cdot z & =x\cdot\left(x\cdot z\right),\\
\left(z\cdot x\right)\cdot x & =z\cdot\left(x\cdot x\right),
\end{align}
for every $x,z\in A$. In fact, the relation between alternative rings
and Moufang planes is a one-to-one correspondence \cite[p. 160]{Pi75}\cite[p. 143]{HP}
and is so deep that Moufang planes are also called ``alternative
planes'' \cite{Ste72}. 

Considering this background, the isomorphism between the Okubic projective
plane and the Cayley plane appears counterintuitive. The Okubic algebra
is non-alternative and non-unital, making it markedly different from
the often-considered octonionic algebra used for realizing the 16-dimensional
Moufang plane. Notably, the Moufang identities in (\ref{eq:MoufangIdent3-1})
do not hold in the Okubo algebra. Hence, the emergence of a Moufang
plane from a projective plane over the Okubo algebra demands a thorough
explanation.

In fact, the Moufang property of a plane is tied to the alternativity
of its associated planar ternary ring (PTR). Typically, this ternary
ring is linear (see below) and thus isomorphic to the algebra from
which the plane is defined. However, as we will show in this section,
the Okubic case is not so simple. Since the Okubic algebra lacks an
identity, it is not a ring and therefore cannot be employed to coordinatize
the Okubic projective plane. Clearly the best candidate for such coordinatization
are the octonions $\mathbb{O}$. 

\subsection{Coordinatizing the Okubic plane with Octonions}

In incidence geometry, relationships between incidence planes and
algebraic structures are developed after a process of relabeling called
coordinatisation that involves a set $\mathscr{C}$ containing the
symbols $0,1\in\mathscr{C}$ and not containing the symbol $\infty$.
In most of cases, dealing with unital algebras the set $\mathscr{C}$
is just the original algebra used for the definition of the plane,
but since the Okubic algebra is not unital the set of symbol $\mathscr{C}$
cannot be $\mathcal{O}$. In our case, a natural candidate for the
set $\mathscr{C}$ is clearly the ring of octonions $\mathbb{O}$.
To coordinatise the Okubic plane with octonionic coordinates we consider
the non-degenerate quadrangle $\diamondsuit=\left\{ \left(0,0\right),\left(e,e\right),\left(0\right),\left(\infty\right)\right\} $
of the Okubic projective plane $\mathcal{O}P^{2}$ and use the ring
of octonions $\mathbb{O}$ for its coordinatisation so that, in the
new coordinates, the quadrangle is $\left\{ \left(0,0\right),\left(1,1\right),\left(0\right),\left(\infty\right)\right\} $.
From the general theory we know \cite[Cor. 3.4]{Fau14} that up to
isomorphism there is a unique standard coordinatisation that maps
$\diamondsuit$ into $\left\{ \left(0,0\right),\left(1,1\right),\left(0\right),\left(\infty\right)\right\} $.
In fact, such coordinatisation is easily obtained sending the Okubic
elements with their respective octonionic representative following
the previous deformation of the product (see Tab. \ref{tab:Oku-Para-Octo}),
so that $e\in\mathcal{O}$ is sent to $1\in\mathbb{O}$ and the other
elements of the base $\left\{ \text{i}_{1},...,\text{i}_{7}\right\} $
in (\ref{eq:definizione i ottonioniche}) are sent into a multiple
of the imaginary unit of the octonions $\mathbb{O}$. 

Once the relabeling process called coordinatization is done, we can
now define a unique ternary ring $\theta\left(s,x,t\right)$ that
encodes algebraically the geometrical properties of the incidence
plane. Indeed, we define the planar ternary ring (PTR) by the incidence
rules of the plane so that 
\begin{equation}
\theta\left(s,x,t\right)=y,\,\,\,\text{iff \,\,}\left(x,y\right)\in\left[s,t\right],
\end{equation}
for all $x,y,s,t\in\mathscr{C}$. For this ternary ring we then define
an \emph{associated product} and an \emph{associated sum}, i.e.
\begin{align}
sx & \coloneqq\theta\left(s,x,0\right),\\
x+t & \coloneqq\theta\left(1,x,t\right).
\end{align}
Then algebraic properties of the associated product and of the associated
sum are then studied in order to deduce geometrical properties of
the coordinatized projective plane. 

Note that in case of the octonionic projective plane $\mathbb{O}P^{2}$
we have that the product and the sum of $\theta$ coincided with those
defined over the algebra of octonions $\left(\mathbb{O},+,\cdot\right)$
so that
\begin{equation}
\theta\left(s,x,t\right)=sx+t=s\cdot x+t.
\end{equation}
When this happen, the planar ternary ring is called\emph{ linear}
\cite[Sec. 22.4]{Compact Projective}. Unfortunately, this is not
the case for the Okubic projective plane so that the ternary ring
derived coordinatising the Okubic projective plane with the ring of
the octonions $\mathbb{O}$ is not linear since

\begin{equation}
\theta\left(s,x,t\right)=sx+t\neq s*x+t,
\end{equation}
as one easily might expect since the octonionic product is not the
Okubic product. 

In summary, while the planar ternary ring is alternative (thereby
not contradicting the one-to-one relationship between Moufang planes
and alternative rings), the originating algebra is non-alternative.
This distinction arises because the ternary ring derived from coordinatization
is nonlinear.
\begin{rem}
It is worth noting that while we are relabeling the Okubic coordinates
with octonions in order to obtain a planar ternary ring, this process
does not constitutes at all a projective isomorphism between the Okubic
plane $\mathcal{O}P^{2}$ and the octonionic projective plane $\mathbb{O}P^{2}$
. Indeed, consider any map $\phi:\mathcal{O}\longrightarrow\mathbb{O}$
is such that
\begin{equation}
\begin{array}{c}
0_{\mathcal{O}}\longrightarrow0_{\mathbb{O}},\\
e_{\mathcal{O}}\longrightarrow1_{\mathbb{O}},\\
x_{\mathcal{O}}\longrightarrow\phi\left(x\right).
\end{array}\label{eq:map}
\end{equation}
Then consider on the octonionic plane $\mathbb{O}P^{2}$ the three
points $\left(0_{\mathbb{O}},0_{\mathbb{O}}\right),\left(1_{\mathbb{O}},1_{\mathbb{O}}\right),\left(\phi\left(x\right),\phi\left(x\right)\right)$.
These three points are collinear, belonging to the same line $\left[1,0\right]=\left\{ x\in\mathbb{O}:\left(x,1\cdot x\right)\right\} $,
i.e. the line passing through the origin with slope $s=1$. Nevertheless,
in the Okubic plane $\mathcal{O}P^{2}$ the three points $\left(0_{\mathcal{O}},0_{\mathcal{O}}\right),\left(e_{\mathcal{O}},e_{\mathcal{O}}\right),\left(x,x\right)$
are not collinear. Indeed the only line joining $\left(0_{\mathcal{O}},0_{\mathcal{O}}\right),\left(e_{\mathcal{O}},e_{\mathcal{O}}\right)$
is $\left[e_{\mathcal{O}},0\right],$i.e. $\left(0_{\mathcal{O}},0_{\mathcal{O}}\right),\left(e_{\mathcal{O}},e_{\mathcal{O}}\right)\in\left[e_{\mathcal{O}},0\right]=\left\{ x\in\mathcal{O}:\left(x,e*x\right)\right\} $,
while the point $\left(x,x\right)\notin\left[e_{\mathcal{O}},0\right]$
since $x\neq e*x$ for any $x\neq e$. We thus have that any relabeling
process of the Okubic algebra with the octonionic algebra does not
yield to a collineation and, in fact, changes the incidence rules
of the plane. 
\end{rem}

\subsection{Direct verification of the ``Little Desargues Theorem'' }

It is well-known that the \textquotedbl\emph{Little Desargues Theorem}\textquotedbl{}
is valid in every Moufang plane. This theorem is a weaker version
of the \emph{Desargues Theorem}, which holds true for every Moufang
plane over an associative algebra.
\begin{figure}
\begin{centering}
\includegraphics[scale=0.85]{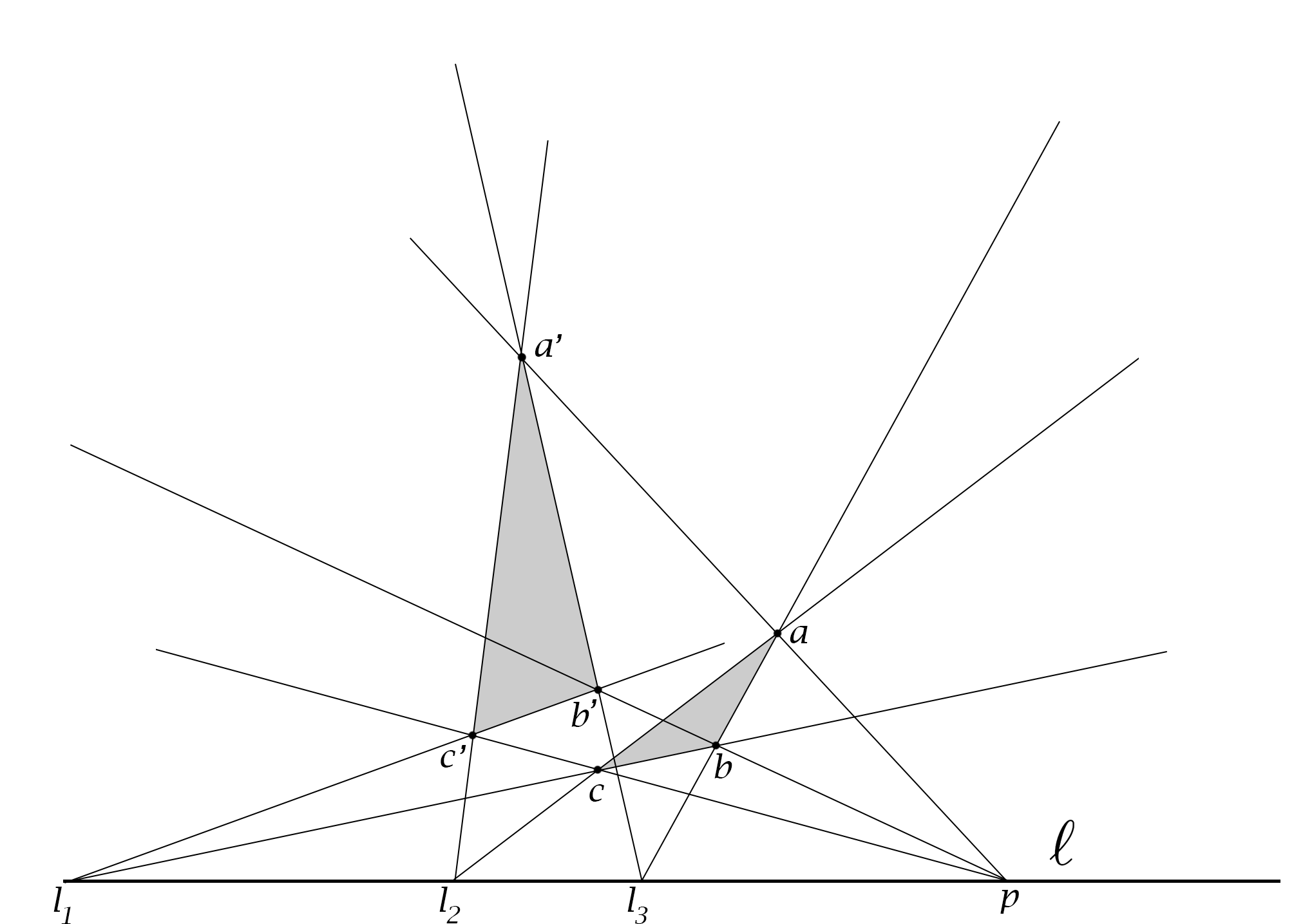}
\par\end{centering}
\caption{\label{fig:Little-Desargues-configuration}Little Desargues configuration:
two triangles $a,b,c$ and $a',b',c'$ are perspective, i.e. lines
$\overline{aa'}$, $\overline{bb'}$ and $\overline{cc'}$ interesect
on the same point $p$ that is the origin of the perspectivity, then
the points of intersection of corresponding sides all lie on one line
$\ell$ that is the axis of the perspectivity. In the Little Desargues
configuration the perspectivity that relates the two triangles is
also an elation, thus the center of the perspectivity $p$ is incident
to the axis $\ell$.}
\end{figure}

The Deasargues theorem states that if two triangles $a,b,c$ and $a',b',c'$,
are perspective, i.e. lines $\overline{aa'}$, $\overline{bb'}$ and
$\overline{cc'}$ interesect on the same point $p$ that is the origin
of the perspectivity, then the points of intersection of corresponding
sides all lie on one line $\ell$, termed the axis of the perspectivity.
However, this theorem is not valid in non-associative Moufang planes.
A special case arises when the point $p$ also lies on the axis $\ell$,
making the perspectivity an elation (see sec. \ref{sec:Elations});
this case is valid in all Moufang planes and is known as the \textquotedbl Little
Desargues Theorem\textquotedbl (see Figure \ref{fig:Little-Desargues-configuration}).

Rather than presenting a formal proof of the theorem's validity (which
is already established for any Moufang plane), we opted for a numerical
verification using a Wolfram Mathematica notebook now available at
the repository \texttt{https://github.com/DCorradetti/OkuboAlgebras}.
The notebook is fully documented with a notation coherent to that
used in this article, so that it can be easily used to verify all
calculations of this article involving octonions, para-octonions and
the real Okubic algebra. To validate the Little Desargues Theorem,
we first represented the real Okubic algebra three-by-three complex
Hermitian matrices endowed with the Okubic product in (\ref{eq:product Ok}).
Next, we developed the following Mathematica functions:
\begin{itemize}
\item \texttt{sLine{[}$x_{1}$,$y_{1}$,$x_{2}$,$y_{2}${]}}: Computes
the slope of the line connecting points $\left(x_{1},y_{1}\right)$
and $\left(x_{2},y_{2}\right)$;
\item \texttt{tLine{[}$x_{1}$,$y_{1}$,$x_{2}$,$y_{2}${]}}: Determines
the offset of the line connecting points $\left(x_{1},y_{1}\right)$
and $\left(x_{2},y_{2}\right)$;
\item \texttt{xPoint{[}$s_{1}$,$t_{1}$,$s_{2}$,$t_{2}${]}}:Calculates
the $x$-coordinate of the intersection point of lines $\left[s_{1},t_{1}\right]$
and $\left[s_{2},t_{2}\right]$;
\item \texttt{yPoint{[}$s_{1}$,$t_{1}$,$s_{2}$,$t_{2}${]}}:Calculates
the $y$-coordinate of the intersection point of lines $\left[s_{1},t_{1}\right]$
and $\left[s_{2},t_{2}\right]$;
\item \texttt{incidence{[}$x$,$y$,$s$,$t${]}}:Determines the $y$-coordinate
of the intersection point of lines $\left[s_{1},t_{1}\right]$ and
$\left[s_{2},t_{2}\right]$;
\end{itemize}
All function arguments are elements of the real Okubo algebra. Then,
to set up the configuration for the Little Desargues Theorem, we: 
\begin{enumerate}
\item Defined the center of the perspectivity $p$ and the axis $\ell$
so that $p\in\ell$.
\item Picked a point $a$ not belonging to $\ell$ and a point $a'$ incident
to the line $\overset{\rightharpoonup}{ap}$.
\item Picked a point $b$ not belonging to $\ell$ nor $\overset{\rightharpoonup}{ap}$
and defined the line $\overset{\rightharpoonup}{ab}$.
\item Found the intersection $l_{3}$of $\overset{\rightharpoonup}{ab}$
with $\ell$.
\item Found the point $b'$ given by the intersection of the lines $\overset{\rightharpoonup}{l_{3}a'}$
and $\overset{\rightharpoonup}{bp}$.
\item Picked a point $c$ not belonging to $\ell$ nor $\overset{\rightharpoonup}{ap}$
nor $\overset{\rightharpoonup}{bp}$ and thus defined the line $\overset{\rightharpoonup}{ac}$.
\item Found the intersection $l_{2}$of $\overset{\rightharpoonup}{ac}$
with $\ell$. 
\item Found the point $c'$ given by the intersection of lines $\overset{\rightharpoonup}{l_{2}a'}$
and $\overset{\rightharpoonup}{cp}$.
\end{enumerate}
Finally, to check the validity of the ``little Desargues theorem''
we verified that the point $l_{1}$, given by the intersection of
$\overset{\rightharpoonup}{cb}$ with $\overset{\rightharpoonup}{c'b'}$,
was incident to $\ell$. As shown in the Wolfram notebook, we verified
the validity of the ``little Desargues theorem'', and, choosing
$p\notin\ell$, that the full Desargues theorem is not valid. 

\section{Conclusions}

This work provides a novel construction of the 16-dimensional Cayley
plane using two flexible algebras: the para-octonions and the real
Okubo algebra. Despite lacking many algebraic properties of octonions,
including alternativity and identity element, both algebras nonetheless
gives the same projective plane up to isometries. Through two explicit
collineations, we established an equivalence between the Okubo, octonionic,
and para-octonionic planes, i.e., $\mathcal{O}P^{2}$, $\mathbb{O}P^{2}$
and $p\mathbb{O}P^{2}$. Moreover, numerical computations directly
confirmed foundational projective properties like the Little Desargues
Theorem.
\begin{table}
\centering{}%
\begin{tabular}{|c|c|c|c|}
\hline 
 & $\mathbb{O}$ & $p\mathbb{O}$ & $\mathcal{O}$\tabularnewline
\hline 
\hline 
Unital & Yes & No & No\tabularnewline
\hline 
Paraunital & No & Yes & No\tabularnewline
\hline 
Alternative & Yes & No & No\tabularnewline
\hline 
Flexible & Yes & Yes & Yes\tabularnewline
\hline 
Composition & Yes & Yes & Yes\tabularnewline
\hline 
Automorphism & $\text{G}_{2}$ & $\text{G}_{2}$ & $\text{SU}\left(3\right)$\tabularnewline
\hline 
\end{tabular}\caption{\label{tab:Summary-Ok-Oct-pOct}Summary of the algebraic properties
of the three division and composition algebras that allows a straightforward
and natural definition of the Cayley plane with the mathematical setup
described in this thesis.}
\end{table}

Among the three constructions of the 16-dimensional Moufang plane,
the one based on the Okubo algebra $\mathcal{O}$ is the simplest
possible for the definition of such a plane, requiring only an 8-dimensional
algebra that is neither alternative nor unital and that has an automorphism
group of real dimension $8$, compared to that of the octonions, and
para-octonios that has dimension $14$ (see Table \ref{tab:Summary-Ok-Oct-pOct}).
Surprisingly, even if historically octonions $\mathbb{O}$ were the
first algebra used for the definition of the compact 16-dimensional
Moufang plane, they are the less economic algebra that can be used
for the definition such plane. For the sake of completeness we should
say that, in order to have an affine plane correctly defined with
a natural setup as that defined above, the 8-dimensional algebra used
for its definition must be a division algebra. Moreover, to have a
completion of the affine plane in correspondence with a Veronese-type
of condition as those above, one need to have a composition algebra.
Thus, for the generalised Hurwitz theorem \cite{ElDuque Comp}, the
only three algebras for which this setup can exist are those recalled
in this paper.

As a corollary of the construction presented in this work, concrete
geometric realization of the real forms of the exceptional Lie groups
$\text{E}_{6\left(-26\right)}$, $\text{F}_{4\left(-52\right)}$ and
$\text{G}_{2\left(-14\right)}$ emerge without recurring to the uses
of octonions. This challenges the conventional thinking that links
exceptional Lie and Jordan structures to octonions and emphasizes
the role of symmetric composition algebras. Future work should further
explore, algebraic and physical interpretations of this new realization. 

From the algebraic point of view, it is known the existence of an
Okubo Jordan algebra \cite{Elduque Gradings SC}, that we expect to
be linked with 16-dimensional Moufang plane as the exceptional Jordan
algebra is \cite{Jordan}. Okubic construction of Tits-Freudenthal
Magic Square were already considered for symmetric composition algebras
\cite{EldMS1,EldMS2}, so we do expect to find a geometrical interpretation
of those construction as Freudenthal and Rosenfeld did for the Hurwitz
version\cite{Freud 1965,Rosenfeld-1993}. From the physical side,
given the connections between M-theory and the octonionic Cayley plane,
we expect that alternative constructions like the Okubo algebra that
has $\text{SU}\left(3\right)$ as automorphism group instead of $\text{G}_{2}$
may unravel novel phases of the theory. In particular, the investigation
of the physical consequences of the lack of unity of Okubo algebra
may turn out to be rather intriguing. We also have provided a valuable
reference implementation of the algebras with examples in a Wolfram
Mathematica notebook in order to help researchers in practical calculations.

All in all, by going beyond the deeply rooted connections between
octonions and exceptional mathematics, this study paves the way to
reimagining non-associative geometry. Symmetric composition algebras,
once considered pathological, may encapsulate geometric worlds equally
rich as their unital cousins. Much work remains to fully chart the
landscape.

\section{Acknowledgments}

We thank Alberto Elduque for useful suggestions and the anonymous
referee of Communications in Algebra for pointing out the isomorphism
of the Okubo projective plane with the octonionic plane. The work
of AM is supported by a \textquotedblleft Maria Zambrano\textquotedblright{}
distinguished researcher fellowship, financed by the European Union
within the NextGenerationEU program.

$*$\noun{ Departamento de Matemática, }\\
\noun{Universidade do Algarve, }\\
\noun{Campus de Gambelas, }\\
\noun{8005-139 Faro, Portugal} 
\begin{verbatim}
a55499@ualg.pt

\end{verbatim}
$\dagger$\noun{ Instituto de Física Teorica, Dep.to de Física,}\\
\noun{Universidad de Murcia, }\\
\noun{Campus de Espinardo, }\\
\noun{E-30100, Spain}
\begin{verbatim}
alessio.marrani@um.es 

\end{verbatim}
$\ddagger$\noun{ Dipartimento di Scienze Matematiche, Informatiche
e Fisiche, }\\
\noun{Università di Udine, }\\
\noun{Udine, 33100, Italy} 
\begin{verbatim}
francesco.zucconi@uniud.it
\end{verbatim}

\end{document}